\documentclass[a4paper,11pt]{amsart}

\usepackage[colorinlistoftodos,bordercolor=orange,backgroundcolor=orange!20,linecolor=orange,textsize=normalsize]{todonotes}

\usepackage{hyperref}
\usepackage{mathrsfs}
\usepackage{mathtools} 
\usepackage{paralist}
\usepackage{phdalgo}
\usepackage{fixmath}
\usepackage{fullpage}
\usepackage{kbordermatrix}
\usepackage{tikz}

\usetikzlibrary{arrows}
\usetikzlibrary{patterns}
\usetikzlibrary{calc}
\usetikzlibrary{shapes}
\usetikzlibrary{positioning}

\usepackage{pgfplots}
\usetikzlibrary{external}

\usepackage{amsmath}
\usepackage{amsthm}
\usepackage{amssymb}
\usepackage{amsfonts}

\renewcommand{\geq}{\geqslant}
\renewcommand{\leq}{\leqslant}

\renewcommand{\succeq}{\succcurlyeq}
\renewcommand{\le}{\leq}
\renewcommand{\ge}{\geq}

\newcommand{\arxiv}[1]{\href{http://arxiv.org/abs/#1}{arXiv:#1}}

\newcommand{\ie}{i.e.}

\DeclareMathOperator{\sign}{\mathsf{sign}}

\DeclareMathAlphabet{\mathbfcal}{OMS}{cmsy}{b}{n}
\DeclareMathAlphabet{\mathbbold}{U}{bbold}{m}{n}

\newcommand{\R}{\mathbb{R}}
\newcommand{\nnR}{\mathbb{R}_{\ge 0}}     
\newcommand{\Q}{\mathbb{Q}}
\newcommand{\N}{\mathbb{N}}
\newcommand{\Rinfty}{\R\cup\{-\infty\}}
\newcommand{\puiseux}{\mathbb{K}}

\newcommand{\nnpuiseux}{\mathbb{K}_{\geq 0}}

\newcommand{\trop}[1][]{\ifthenelse{\equal{#1}{}}{ \mathbb{T} }{ \mathbb{T}(#1) }}
\newcommand{\strop}[1][]{{\trop[#1]}_{\pm}}    
\newcommand{\postrop}[1][]{{\trop[#1]}_{+}}    
\newcommand{\negtrop}[1][]{{\trop[#1]}_{-}}    

\newcommand{\tplus}{\oplus}  
\newcommand{\tsum}{\bigoplus}
\newcommand{\tdot}{\odot}
\newcommand{\tminus}{\ominus}  
\newcommand{\abs}[1]{|{#1}|}

\newcommand{\zero}{{-\infty}}

\DeclareMathOperator*{\val}{\mathsf{val}}
\DeclareMathOperator*{\lc}{\mathsf{lc}}
\DeclareMathOperator*{\sval}{\mathsf{sval}}

\newcommand{\bo}[1]{\mathbold{#1}}

\newcommand{\Smpg}{\textsc{Smpg}}
\newcommand{\Smpgi}{\textsc{Smpg($k$)}}
\newcommand{\Smpgcomp}{\textsc{Smpg-comp}}
\newcommand{\Tropicalsdfp}{\textsc{Tmsdfp}}

\newcommand{\spectra}{\mathcal{S}}
\newcommand{\bspectra}{\mathbfcal{S}}

\newtheorem{theorem}{Theorem}
\newtheorem{proposition}[theorem]{Proposition}
\newtheorem{corollary}[theorem]{Corollary}
\newtheorem{lemma}[theorem]{Lemma}
\theoremstyle{definition}
\newtheorem{definition}[theorem]{Definition}
\newtheorem{assumption}[theorem]{Assumption}

\theoremstyle{remark}
\newtheorem{remark}[theorem]{Remark}
\newtheorem{example}[theorem]{Example}

\tikzset{grid/.style={gray!30,very thin}}
\tikzset{axis/.style={gray!50,->,>=stealth'}}
\tikzset{convex/.style={draw=none,fill=lightgray,fill opacity=0.7}}
\tikzset{convexborder/.style={very thick}}
\tikzset{point/.style={blue!50}}
\tikzset{hs/.style={fill opacity=0.3,fill=orange,draw=none}}
\tikzset{hsborder/.style={orange,ultra thick,dashdotted}}

\newcommand{\overbar}[1]{\mkern 1.5mu\overline{\mkern-1.5mu#1\mkern-1.5mu}\mkern 1.5mu}

\newcommand{\loew}{\succeq}

\newcommand{\Pbb}{\mathbb{P}}
\newcommand{\E}{\mathbb{E}}

\newcommand{\Poly}{\mathrm{Poly}}
\newcommand{\ind}{\mathbf{1}}
\newcommand{\shapley}{F}

\newcommand{\NP}{\mathrm{NP}}
\newcommand{\coNP}{\mathrm{coNP}}

\newcommand*\minstate[1]{\tikz[baseline=(char.base)]{
            \node[shape=circle,draw,inner sep=1.1pt] (char) {#1};}}
\newcommand*\maxstate[1]{\tikz[baseline=(char.base)]{
            \node[shape=rectangle,draw,inner sep=2.1pt] (char) {#1};}}

\newcommand*\minstatesmall[1]{\minstate{\scriptsize{#1}}}
\newcommand*\maxstatesmall[1]{\maxstate{\scriptsize{#1}}}

\newcommand{\gameval}{\chi}

\newcommand{\chainspace}{E}
\newcommand{\state}{u}
\newcommand{\stateII}{w}
\newcommand{\payoff}{r}
\newcommand{\avpayoff}{g}
\newcommand{\reclass}{C}
\newcommand{\stdist}{\pi}
\newcommand{\trmatrix}{P}
\newcommand{\card}[1]{|{#1}|}
\newcommand{\dunion}{\uplus}

\newcommand{\pert}{\lambda}
\newcommand{\pertspectra}{\spectra_{\lambda}}

\newcommand{\dominion}{D}
\newcommand{\bias}{v}
\newcommand{\eigenval}{\eta}
\newcommand{\defor}{\delta}

\pgfdeclareplotmark{square**}
{%
  \pgfpathrectangle{\pgfpoint{-.1\pgfplotmarksize}{-\pgfplotmarksize}}{\pgfpoint{.2\pgfplotmarksize}{2\pgfplotmarksize}}%
  \pgfusepathqfill
}

\title{Solving generic nonarchimedean semidefinite programs\\using stochastic game algorithms}
\date{\today}

\thanks{This manuscript has been accepted for publication in Journal of Symbolic Computation, \url{https://doi.org/10.1016/j.jsc.2017.07.002}. \copyright 2018. The present file is made available under the CC-BY-NC-ND 4.0 license \url{http://creativecommons.org/licenses/by-nc-nd/4.0/}. The three authors were partially supported by the ANR projects CAFEIN (ANR-12-INSE-0007) and MALTHY (ANR-13-INSE-0003), by the PGMO program of EDF and Fondation Math\'ematique Jacques Hadamard, and by
the ``Investissement d'avenir'', r{\'e}f{\'e}rence ANR-11-LABX-0056-LMH,
LabEx LMH. M.~Skomra is supported by a grant from R{\'e}gion Ile-de-France.}

\author{Xavier {A}llamigeon}
\author{St{\'e}phane {G}aubert}
\author{Mateusz Skomra}
\address{INRIA and CMAP, \'Ecole Polytechnique, CNRS, Universit{\'e} Paris--Saclay, 91128 Palaiseau Cedex France}
\email{firstname.lastname@inria.fr}

\catcode`\@=11
\@namedef{subjclassname@2010}{%
  \textup{2010} Mathematics Subject Classification}
\catcode`\@=12

\begin{document}

\begin{abstract}
A general issue in computational optimization is to develop combinatorial algorithms for semidefinite programming.
We address this issue when the base field is {\em nonarchimedean}.
We provide a solution for a class of semidefinite
feasibility problems given by generic matrices. Our approach is based on tropical geometry. It relies on tropical spectrahedra, which are defined 
as the images by the valuation of nonarchimedean
spectrahedra. We establish a correspondence between generic tropical spectrahedra and zero-sum stochastic games with perfect information.
The latter have been well studied in algorithmic game
theory. This allows us to solve
nonarchimedean semidefinite feasibility problems
using algorithms for stochastic games. These
algorithms are of a combinatorial nature
and work for large instances.
\end{abstract}

\maketitle

\section{Introduction}
Semidefinite programming consists in optimizing a linear function over
a spectrahedron. The latter is a subset of $\R^n$ defined by linear matrix
inequalities, i.e., a set of the form
\[
\spectra=\{x\in \R^n\colon Q^{(0)}+ x_1 Q^{(1)}+ \dots + x_n Q^{(n)} \loew 0\}
\]
where the $Q^{(k)}$ are symmetric matrices of order $m$, and $\loew$ denotes the
Loewner order on the space of symmetric matrices. By definition, $X\loew Y$
if and only if $X-Y$ is positive semidefinite. 

Semidefinite programming is a fundamental tool in convex optimization.
It is used to solve various applications from engineering sciences, and
also to
obtain approximate solutions or bounds for hard problems arising in combinatorial optimization and semialgebraic
optimization. We refer the reader to~\cite{siam_parrilo}
and~\cite{matousekbook} for information.

Semidefinite programs 
are usually solved via interior point methods. 
The latter provide an approximate solution in a polynomial
number of iterations, provided that a strictly feasible
initial solution, i.e., a point belonging to the interior of the set $\spectra$, is known. We refer the reader to~\cite{deklerk_vallentin}
for a detailed analysis of the complexity in the bit model of interior point
methods for semidefinite programming, and for a discussion
of earlier complexity results based on the ellipsoid method.

Semidefinite programming becomes a much harder matter if one
requires an exact solution. The feasibility problem 
(deciding the emptiness of the set $\spectra$)
belongs to $\NP_{\R} \cap \coNP_{\R}$, where the subscript
$\R$ refers to the BSS model of computation~\cite{ramana_exact_duality}. It is not known
to be in $\NP$ in the bit model. A difficulty here is that
all feasible points may have entries of absolute value doubly
exponential in the size of the input. Also, there may be
no rational solution~\cite{scheiderer}. Beyond their theoretical interest,
exact algorithms for semidefinite programming may be useful to address problems
of formal proofs, which sometimes lead to challenging 
(degenerate) instances~\cite{monniaux_corbineau}. 
Known exact methods rely either
on general purpose semialgebraic techniques (like cylindrical decomposition)
or on dedicated methods based on the computation of critical points, see
the recent work~\cite{exact_lmi} and the references therein. 

Semidefinite programming is meaningful in any real closed field,
and in particular in {\em nonarchimedean} (real closed) fields like the field
of Puiseux series with real coefficients. Such nonarchimedean
semidefinite programming problems arise when considering {\em parametric} semidefinite programming problems over the reals, or structured problems
in which the entries of the matrices have different
orders of magnitudes. They are also of an intrinsic
interest, since, by analogy with the situation in linear programming~\cite{megiddo_lp},
shifting to the nonarchimedean case is expected to shed light on
the complexity of the classical problem over the reals. 
We note that the fields of Puiseux series are representative of the general nonarchimedean situation, since any nonarchimedean  ordered field can be embedded in a field of generalized
(Hahn) formal series~\cite[Th.~5.2.20]{steinberg}.

\subsection*{Description of main results}

We address semidefinite programming in the nonarchimedean case, to which methods from tropical geometry can be applied.
These methods are expected to allow one, in generic situations,
to reduce semialgebraic problems to combinatorial 
problems, involving only the nonarchimedean valuations (leading exponents)
of the coefficients of the input.

We exploit the tropical approach by considering tropical spectrahedra,
defined as the images by the valuation of nonarchimedean spectrahedra. 
Yu~\cite{yu_semidefinite_cone} showed that 
the image of the cone of semidefinite matrices
is obtained by tropicalizing the nonnegativity conditions
of minors of order $1$ and $2$.
More generally, we studied the tropicalization of spectrahedra in the companion paper~\cite{tropical_spectrahedra}. We showed that, 
under a genericity condition, tropical spectrahedra are described by explicit inequalities still arising from the nonnegativity of tropical minors of order $1$ and $2$. We recall these results in Section~\ref{section:spectrahedra}. Moreover, we show that we can reduce the tropical semidefinite feasibility problem to the subproblem in which tropical spectrahedra are defined by matrices with a Metzler sign pattern.

In this paper, we show (Theorem~\ref{theorem:simple_games}) that the feasibility problem for a generic tropical spectrahedron is equivalent to solving a stochastic mean payoff game (with perfect information). The complexity of these games is a long-standing open problem. They are not known to be polynomial, however they belong to the class $\NP \cap \coNP$, and they can be solved efficiently in practice. 

This allows us to apply stochastic game algorithms to solve nonarchimedean
semidefinite feasibility problems
(Section~\ref{sec-algo}). We obtain in this way both theoretical
bounds and a practicable method which solves some large scale
instances.  

\subsection*{Related work}

This work originates from the equivalence between
tropical polyhedral feasibility problems and deterministic zero-sum games with mean payoff, established by Akian, Gaubert, and Guterman~\cite{polyhedra_equiv_mean_payoff}. 
The novelty here is the handling of nonlinear semialgebraic convex
problems, and the proof that the well-known class of
{\em stochastic} mean payoff games
correspond precisely to semidefinite feasibility problems with a Metzler
structure, therefore relating two classes of problems which both have
an unsettled complexity.

Moreover, as mentioned above, the computation of exact solutions of semidefinite programming problems is of current interest. In particular, Nie, Ranestad, and Sturmfels~\cite{algebraic_degree_sdp} provided
complexity measures based on the notion of algebraic degree,
and dedicated algorithms have been developed
by Henrion, Naldi, and Safey El Din~\cite{exact_lmi,naldi_rank_constrained}. 

\section{Statement of the problem and illustration of our approach}\label{section:illustration}

A convenient choice of nonarchimedean structure, which we make in this paper, is the \emph{field $\puiseux$ of (absolutely convergent generalized real) Puiseux series}, which is composed of functions of a real positive parameter $t$ of the form
\begin{equation}
\bo x = \sum_{i = 1}^{\infty} c_{\lambda_{i}}t^{\lambda_{i}} \, , \label{eq:series}
\end{equation}
where $(\lambda_{i})_{i \ge 1}$ is a strictly decreasing sequence of real numbers that is either finite or unbounded, $c_{\lambda_{i}} \in \R \setminus \{ 0\}$ for all $\lambda_{i}$, and the latter series is required to be absolutely convergent
for $t$ large enough. There is also a special, empty series, which is denoted by $0$. 
We denote by $\lc(\bo x)$ the coefficient $c_{\lambda_{1}}$ of the leading term in the series $\bo x$, with the convention that $\lc(0) = 0$. The addition and multiplication in $\puiseux$ are defined in a natural way. Moreover, $\puiseux$ can be endowed with a linear order $\ge$, which is defined as $\bo x \ge \bo y$ if $\lc(\bo x - \bo y) \ge 0$.  We denote $\nnpuiseux$ the set of nonnegative series $\bo x$, i.e.,~satisfying $\bo x \geq 0$. The \emph{valuation} of an element $\bo x \in \puiseux$ as in~\eqref{eq:series} is defined as the greatest exponent $\lambda_{1}$ occurring in the series. Equivalently, the valuation is given by 
\[
\lim_{t \to +\infty} \log_t \abs{\bo x(t)} \, ,
\]
where $\log_t(z) \coloneqq \log(z) / \log(t)$. Up to a change of variables, $\puiseux$ coincides with the field of {\em generalized Dirichlet series} introduced by Hardy and Riesz~\cite{hardy}. Indeed, ordinary Dirichlet series are obtained by setting $t \coloneqq \exp(s)$ and $\lambda_i \coloneqq -\log i$. 
Van den Dries and Speissegger showed, among other results,
that $\puiseux$ is real closed~\cite[Corollary~9.2 and 
Section~10.2]{van_dries_power_series}.
The familiar field of ordinary, convergent Puiseux series, is obtained by requiring the sequence $\lambda_i$ to
consist of rational numbers included in an arithmetic progression.
Working with the larger field $\puiseux$, whose value
group is $\R$ instead of $\Q$, leads to more transparent results. 

In this paper, we consider the semidefinite feasibility problem over the field $\puiseux$. It is more convenient to start with the homogeneous case. More precisely, given symmetric matrices $\bo Q^{(1)}, \dots, \bo Q^{(n)} \in \puiseux^{m \times m}$, we focus on the problem of determining whether or not the following spectrahedral cone
\[
\{ \bo x \in \nnpuiseux^n \colon \bo x_1 \bo Q^{(1)} + \dots + \bo x_n \bo Q^{(n)} \loew 0 \} \, 
\]
is \emph{trivial}, meaning that it is reduced to the identically null point of $\nnpuiseux^n$. The linear map $\bo x \mapsto \bo x_1 \bo Q^{(1)} + \dots + \bo x_n \bo Q^{(n)} $ is a \emph{matrix pencil}, which we denote by $\bo Q(\bo x)$ for more brevity. The aforementioned decision problem is in fact equivalent to the nontriviality problem of general spectrahedral cones in which the nonnegativity condition $\bo x \in \nnpuiseux^n$ is relaxed. Indeed, given a pencil $\bo Q(\bo x)$ of symmetric matrices, the spectrahedron $\{\bo x \in \puiseux^n \colon \bo Q(\bo x) \loew 0\}$ is trivial if and only if the spectrahedron $\{(\bo y, \bo z) \in \nnpuiseux^{2n} \colon (\bo y_1 - \bo z_1) \bo Q^{(1)} + \dots + (\bo y_n - \bo z_n) \bo Q^{(n)} \loew 0 \}$ is trivial. Even if the instances arising in this way are unlikely to be generic in the sense we discuss later in Section~\ref{section:spectrahedra}, we consider that the decision problem which we focus on already retains much of the complexity of the semidefinite feasibility problem over the field of Puiseux series. As we shall see in Section~\ref{sec:nonconic}, our method can be extended to handle the feasibility problem for affine spectrahedra given by the relation $\{\bo x \in \nnpuiseux^{n} \colon \bo Q^{(0)} +  \bo x_1 \bo Q^{(1)} + \dots + \bo x_n \bo Q^{(n)} \loew 0 \}$.

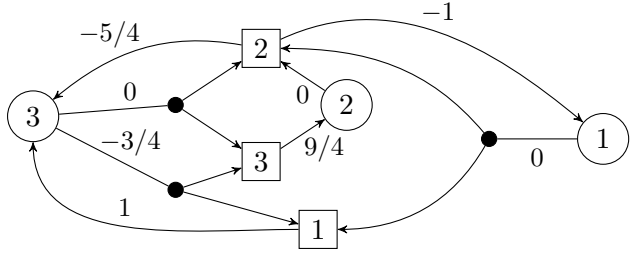
\begin{figure}[t]
\centering
\begin{tikzpicture}[scale=0.75,>=stealth',max/.style={draw,rectangle,minimum size=0.5cm},min/.style={draw,circle,minimum size=0.5cm},av/.style={draw, circle,fill, inner sep = 0pt,minimum size = 0.2cm}]

\node[min] (min3) at (-5, 0.8) {$3$};
\node[min] (min1) at (5, 0.4) {$1$};
\node[min] (min2) at (0.5, 1) {$2$};

\node[max] (max1) at (0, -1.2) {$1$};
\node[max] (max2) at (-1, 2) {$2$};
\node[max] (max3) at (-1, 0) {$3$};

\node[av] (av13) at (-2.5,-0.5){};
\node[av] (av23) at (-2.5, 1){};
\node[av] (av12) at (3, 0.4){};

\draw[->] (max3) to node[below right=-1ex, font=\small]{$9/4$} (min2);
\draw[->] (min2) to node[below, font=\small]{$0$} (max2);
\draw[->] (max2) to[out = 170, in = 40] node[above left=-1ex, font=\small]{$-5/4$} (min3);
\draw[->] (max2) to[out = 25, in = 140] node[above, font=\small]{$-1$} (min1);

\draw[-] (min3) to node[above right=-0.7ex and -1.8ex, font=\small]{$-3/4$} (av13);
\draw[-] (min3) to node[above right, font=\small]{$0$} (av23);
\draw[-] (min1) to node[below, font=\small]{$0$} (av12);

\draw[->] (max1) to[out = 180, in = -90] node[above, font=\small]{$1$} (min3);

\draw[->] (av13) to (max1);
\draw[->] (av13) to (max3);
\draw[->] (av23) to (max2);
\draw[->] (av23) to (max3);
\draw[->] (av12) to[out=-120, in = 0] (max1);
\draw[->] (av12) to[out=130, in = 0] (max2);

\end{tikzpicture}
\caption{Stochastic game associated with the spectrahedron considered in Section~\ref{section:illustration}.}\label{fig:example_of_a_game}
\end{figure}

Our approach is best explained when the off-diagonal entries of the 
matrices
$\bo Q^{(1)}$, \dots, $\bo  Q^{(n)}$
are nonpositive. We associate to these matrices
the following zero-sum game.
There are two players, Player Min and Player Max, who
control disjoint sets of states. 
The states of Player Min can be identified to the variables $\bo x_1,\dots,\bo x_n$. The states of Player Max can be identified to the row (or column) indices of the matrices, i.e., to the elements of $\{1,\dots,m\}$.
In state $\bo x_k$, Player Min chooses a subset $\{i,j\}\subset \{1,\dots,m\}$ 
such that the entry $\bo Q^{(k)}_{ij}$ is negative. Next, Player Min pays to Player
Max the opposite of the valuation of $\bo Q^{(k)}_{ij}$. Then,
Nature selects one element among $i$ or $j$ at random, 
with uniform probabilities (i.e., $i$ or $j$ is drawn with 
probability $1/2$).\footnote{We allow the case $i = j$. In other words, Player Min can choose a subset $\{ i\} \subset \{1,\dots,m\}$ such that $\bo Q^{(k)}_{ii}$ is negative. In this case, Nature selects $i$.} If $i$ is drawn, meaning that the current state
is now $i$, Player Max chooses
a variable $\bo x_l$ such that $\bo Q^{(l)}_{ii}$ has a positive sign. He receives the valuation of $\bo Q^{(l)}_{ii}$ from Player Min,
and the next state becomes $\bo x_l$. 
If $j$ is drawn, the rule of move and the payment are identical, up to the replacement of $i$ by $j$. 
A (stationary) policy of one player is a map which associates to a state
of the player an admissible move.
We are interested in infinite plays, with an infinite number of turns.
Player Max looks for a policy which maximizes
the mean payment received from Player Min per time unit, while Player Min
looks for a policy which minimizes it.

We can think informally of this construction as follows:
Player Min wishes to show
that the semidefinite programming problem is infeasible,
whereas Player Max wishes to show that it is feasible. 

Let us now illustrate our approach on an example in dimension $3$. We consider the following pencil of symmetric matrices
\[
\bo Q(\bo x) \coloneqq
\begin{bmatrix}
t \bo x_{3} & - \bo x_{1} & -t^{3/4} \bo x_{3}\\
-\bo x_{1} & t^{-1} \bo x_{1} + t^{-5/4} \bo x_{3} - \bo x_{2} & - \bo x_{3}\\
-t^{3/4}\bo x_{3} & - \bo x_{3} & t^{9/4} \bo x_{2}
\end{bmatrix} \, ,
\]
and we aim at checking whether or not the spectrahedron $\{ \bo x \in \nnpuiseux^3 \colon \bo Q (\bo x) \loew 0 \}$ is trivial. 
The associated game
is depicted in Figure~\ref{fig:example_of_a_game}.
The states of Player Min are depicted by circles.
The states of Player Max are depicted by squares.
The states in which Nature plays are depicted by full dots.
The admissible moves of the game are represented by the edges between the states. The corresponding payments received by Player Max are indicated on these edges.

Observe that both players in this game have only two policies: at state \maxstate{2} Player Max can choose the move that goes to $\minstate{1}$ or the one that goes to $\minstate{3}$, whereas at state \minstate{3} Player Min can choose the move $\bigl\{ \maxstate{1}, \maxstate{3} \bigr\}$ or the move $\bigl\{ \maxstate{2}, \maxstate{3} \bigr\}$.

Suppose that Player Max chooses the policy which goes to state $\minstate{1}$ from state $\maxstate{2}$. If Player Min plays using the move $\bigl\{ \maxstate{1}, \maxstate{3} \bigr\}$ at state $\minstate{3}$, then a standard computation on Markov chain 
(which we present in Example~\ref{ex:payoff_computation}) 
shows that the long-term average payoff of Player Max is equal to $3/40$. Similarly, if Player Min chooses the move $\bigl\{ \maxstate{2}, \maxstate{3} \bigr\}$ instead, then the payoff of Player Max is equal to $1/56$. Therefore, Player Max has a policy which guarantees him to win the game, \ie, to obtain a nonnegative payoff. The main theorem of this paper states that this fact is equivalent to the nontriviality of the spectrahedron~$\{ \bo x \in \nnpuiseux^3 \colon \bo Q (\bo x) \loew 0 \}$. As we shall see, a winning policy of Player Max (resp.\ Min) provides a feasibility (resp.\ infeasibility) certificate.
The mean payoff represents a feasibility/infeasibility margin.

\section{Preliminary results on tropical spectrahedra}\label{section:spectrahedra}

\subsection{Tropical algebra}\label{section:tropical_algebra}

In this section, we recall some basic concepts of tropical algebra.

The \emph{tropical semifield} $\trop$ is the set
$\Rinfty$, endowed with the addition
$x \tplus y \coloneqq \max(x,y)$ and the multiplication $x \tdot y \coloneqq x + y$. 
The term semifield refers to the fact that the addition does not have
an opposite law.
We use the notation $\tsum_{i = 1}^{n}a_{i} = a_{1} \tplus \dots \tplus a_{n}$ and $a^{\tdot n} = a \tdot \dots \tdot a$ ($n$ times). We also endow $\trop$ with the standard order $\le$. 
The reader may consult~\cite{butkovic,maclagan_sturmfels} for more information on the tropical semifield.

We denote by $\val \colon \puiseux \to \Rinfty$ the map which associates a Puiseux series $\bo x \in \puiseux$ to its valuation. We use the convention $\val(0) = -\infty$.  It is immediate to see that the map $\val$ satisfies the following properties
\begin{align}
\val(\bo x + \bo y) &\leq \max( \val (\bo x ), \val (\bo y))
\label{e-val}\\
\qquad 
\val(\bo x  \bo y) &= \val (\bo x ) + \val (\bo y)
\label{e-val2}
\end{align}
meaning that $\val$ is a {\em nonarchimedean valuation}. Moreover, the equality
holds in~\eqref{e-val} if the leading terms of $\bo x$ and $\bo y$ do not cancel, which is the case if $\val (\bo x ) \neq \val (\bo y)$ or if
$\bo x, \bo y\geq 0$.
In particular, the map $\val$ yields an order-preserving
morphism of semifields from $\nnpuiseux$ to $\trop$. 

 The \emph{sign} of a series $\bo x\in \puiseux$ is equal to $+1$ if $\bo x >0$, $-1$ if $\bo x<0$, and $0$ otherwise. 
The \emph{signed valuation}, denoted by $\sval$,
associates with a series $\bo x\in\puiseux$ the couple
$(\sign (\bo x), \val (\bo x))$.
We denote by $\strop$ the
image of $\puiseux$ by $\sval$. We call it the set
of \emph{signed tropical numbers}. For brevity,
we denote an element of the form $(\epsilon,a)$ by $a$
if $\epsilon =1$, $\tminus a$ if $\epsilon =-1$,
and $-\infty$ if $\epsilon =0$. Here, $\tminus$ is a formal
symbol. 
 We call the elements of the first and second kind
the \emph{positive} and \emph{negative} tropical numbers, respectively.
We denote by $\postrop$ and $\negtrop$ the corresponding sets.
In this way, $(-2)$ is tropically positive, but $\tminus (-2)$ is tropically negative. Also, $\trop$ is embedded in $\strop$, i.e., $\trop = \postrop \cup \{ \zero\}$. 
In $\strop$, we define a \emph{modulus} function, $\abs{\cdot} \colon \strop \to \trop$, as $\abs{\zero} = \zero$ and $\abs{a} = \abs{\tminus a} = a$ for all $a \in \postrop$. We point out that $\tdot$ straightforwardly extends
to $\strop$ using the standard rules for the sign, for instance $2\tdot (\tminus 3)= \tminus 5$. In contrast, we only partially extend the tropical addition $\tplus$ to elements of $\strop$ of identical sign, e.g., $2 \tplus 3 = 3$ and $(\tminus 2) \tplus (\tminus 3) = \tminus 3$.
It is possible to embed $\strop$ in an idempotent semiring, called
the symmetrized tropical semiring~\cite{guterman}, so that the tropical addition
becomes defined for all elements of $\strop$.
Alternatively, this addition may be defined by working in the setting
of hyperfields~\cite{virohyperfields,connesconsani,baker}. 
Here, we shall only use $\strop$ 
as a convenient notation, and we do not rely on an algebraic structure on $\strop$.

We shall extend the valuation maps $\val$ and $\sval$ to vectors and matrices in a coordinate-wise manner.

Finally, we use the notion of tropical polynomials. A \emph{tropical (signed) polynomial} over the variables $X_1, \dots, X_n$ is a formal expression of the form
\begin{align}
P(X) = \tsum_{\alpha \in \Lambda} a_{\alpha} \tdot X_{1}^{\tdot \alpha_{1}} \tdot  \dots \tdot X_{n}^{\tdot \alpha_{n}} \, ,\label{e-def-P}
\end{align}
where $\Lambda \subset \{0, 1, 2, \dots \}^{n}$, and $a_{\alpha} \in \strop \setminus \{\zero\}$ for all $\alpha \in \Lambda$. 
We say that the tropical polynomial $P$ \emph{vanishes} on the point $x\in \strop^n$ if the terms $a_{\alpha} \tdot x_{1}^{\tdot \alpha_{1}} \tdot  \dots \tdot x_{n}^{\tdot \alpha_{n}}$ which have the greatest modulus do not have the same sign. If $P$ does not vanish on $x$, we define $P(x)$ as the tropical sum of the terms which have the greatest modulus. As an example, if $P(X)=2\tdot X_1^{\tdot 3}\tdot X_2^{\tdot 4}\tplus (\tminus 0) \tdot X_2$, 
then $P(1,\tminus 5)=25$, $P(1,-5)= \tminus (-5)$, whereas 
$P$ vanishes on $(1,-5/3)$. These definitions are motivated by the following
immediate observation.
Suppose that
\begin{equation} 
\bo P(X)= \sum_{\alpha \in \Lambda} \bo a_{\alpha} X_{1}^{\alpha_{1}} \dots
X_{n}^{\alpha_{n}} \in \puiseux[X_1,\dots,X_n] \label{eq:polynomial}
\end{equation}
and let $P$ be defined as in~\eqref{e-def-P} with 
$a_\alpha\coloneqq\sval (\bo a_\alpha)$. 
Then, for all $\bo x\in \puiseux^n$,
\[
\sval (\bo P(\bo x)) = P(\sval (\bo x)) \, ,
\]
provided that $P$ does not vanish on $\sval (\bo x)$. 

Given a polynomial $\bo P$ as in~\eqref{eq:polynomial}, we denote by $\bo P^+$ the polynomial formed by the terms $\bo a_{\alpha} X_{1}^{\alpha_{1}} \dots
X_{n}^{\alpha_{n}}$ such that $\bo a_\alpha > 0$. Similarly, $\bo P^-$ refers to the polynomial consisting of the terms $-\bo a_{\alpha} X_{1}^{\alpha_{1}} \dots X_{n}^{\alpha_{n}}$ verifying $\bo a_\alpha < 0$. In this way, $\bo P = \bo P^+ - \bo P^-$. We also use the analogues of these polynomials in the tropical setting. If $P$ is the tropical polynomial given in~\eqref{e-def-P}, we define $P^+$ (resp.\ $P^-$) as the tropical polynomial generated by the terms $\abs{a_\alpha} \tdot X_{1}^{\tdot \alpha_{1}} \tdot  \dots \tdot X_{n}^{\tdot \alpha_{n}}$ where $a_\alpha \in \postrop$ (resp.\ $\negtrop$). Observe that the quantities $P^+(x)$ and $P^-(x)$ are well defined for all $x \in \trop^n$, since the tropical polynomials $P^+$ and $P^-$ only involve tropically positive coefficients.

Throughout the paper, we denote the set $\{1, \dots, k\}$ by~$[k]$.

\subsection{Tropicalization of nonarchimedean spectrahedra} We now discuss the class of tropical spectrahedra. These objects were introduced in our companion paper~\cite{tropical_spectrahedra}. This paper also contains the detailed proofs of all the results presented in this section.

\begin{definition}
A set $\spectra \subset \trop^n$ is said to be a \emph{tropical spectrahedron} if there exists a spectrahedron $\bspectra \subset \nnpuiseux^{n}$ such that $\spectra = \val(\bspectra)$.
\end{definition}
If $\spectra = \val(\bspectra)$, then we refer to $\spectra$ as the \emph{tropicalization} of the spectrahedron $\bspectra$, and $\bspectra$ is said to be a \emph{lift (over the field $\puiseux$)} of~$\spectra$. Checking the triviality of a spectrahedral cone $\bspectra \subset \nnpuiseux^n$ is equivalent to determining whether or not the corresponding tropical spectrahedron $\spectra = \val(\bspectra)$ is \emph{trivial}, \ie, is reduced to the singleton $\{(-\infty, \dots, -\infty)\}$. Therefore, it is convenient to exploit some explicit description of the set $\spectra$. As stated in Theorem~\ref{theorem:generic_metzler}, such a description can be obtained from the tropical minors of order $1$ and $2$, provided that the valuation of the coefficients of the matrices $\bo Q^{(k)}$ is generic. To this purpose, given $i, j \in [m]$, we denote by $Q_{ij}(X)$ the tropical polynomial:
\[
Q_{ij}(X) \coloneqq Q^{(1)}_{ij} \tdot X_{1} \tplus \dots \tplus Q^{(n)}_{ij} \tdot X_{n} \, .
\]
\begin{definition}\label{def:spectrahedron}
Let $Q^{(1)}, \dots, Q^{(n)} \in \strop^{m \times m}$ be symmetric tropical matrices. We introduce the set $\spectra(Q^{(1)}, \dots, Q^{(n)})$ (or simply $\spectra$) of points $x \in \trop^{n}$ that fulfill the following two conditions:
\begin{compactitem}
\item for all $i \in [m]$, $Q_{ii}^{+}(x) \ge Q_{ii}^{-}(x)$;
\item for all $i,j \in [m]$, $i <j$, we have $Q_{ii}^{+}(x) \tdot Q_{jj}^{+}(x) \ge (Q^+_{ij}(x) \tplus Q^-_{ij}(x))^{\tdot 2}$ or $Q_{ij}^{+}(x) = Q_{ij}^{-}(x)$.
\end{compactitem}
\end{definition}

Given $d \geq 1$, the \emph{support} of a point $y \in \trop^d$ is defined as the set of indices $k \in [d]$ such that $y_k \neq \zero$. Given a nonempty subset $K \subset [d]$, and a set $Y \subset \trop^d$, we define the \emph{stratum of $Y$ associated with $K$} as the subset of $\R^K$ formed by the projection $(y_k)_{k \in K}$ of the points $y \in Y$ with support~$K$. We say that a set $Y \subset \trop^{d}$ is \emph{negligible} if every stratum of $Y$ has Lebesgue measure zero.

\begin{theorem}[{\cite[Theorems~32 and~37]{tropical_spectrahedra}}]\label{theorem:generic_metzler}
There exists a negligible set $Y \subset \trop^{d}$ with $d = nm(m+1)/2$ such that the following property holds. Suppose that $\bo Q^{(1)}, \dots, \bo Q^{(n)} \in \puiseux^{m \times m}$ is a sequence of symmetric matrices and let $\bspectra = \{ \bo x \in \nnpuiseux^n \colon \bo Q(\bo x) \loew 0 \}$ be the associated spectrahedron. Denote $Q^{(k)} = \sval(\bo Q^{(k)})$ for all $k$. If the vector with entries $\abs{Q^{(k)}_{ij}}$ (for $i \le j$) does not belong to $Y$, then we have
\[
\val(\bspectra) = \spectra(Q^{(1)}, \dots, Q^{(n)}) \, .
\]
\end{theorem}
The set $Y$ from Theorem~\ref{theorem:generic_metzler} can be constructed explicitly, as a finite union of hyperplanes. Every hyperplane arises from a condition expressing the absence of a flow in a certain directed hypergraph associated with $\spectra(Q^{(1)}, \dots, Q^{(n)})$. This construction is detailed in \cite[Section~5.4]{tropical_spectrahedra}. Nevertheless, the number of directed hypergraphs (and subsequently, of hyperplanes) to be considered is exponential. For this reason, we do not give an explicit description of $Y$.

One important special case is when the matrices $\bo Q^{(k)}$ are Metzler matrices. Recall that a matrix is called \emph{(negated) Metzler matrix} if its off-diagonal coefficients are nonpositive. Similarly, we say that a matrix $M \in \strop^{m \times m}$ is a \emph{tropical Metzler matrix} if $M_{ij} \in \negtrop \cup \{\zero\}$ for all $i \neq j$. Under the assumption that the matrices $Q^{(k)}$ are tropical Metzler matrices, Definition~\ref{def:spectrahedron} gets slightly simpler:

\begin{lemma}\label{lemma:metzler_spectrahedra}
Suppose that the matrices $Q^{(1)}, \dots, Q^{(n)} \in \strop^{m \times m}$ are symmetric tropical Metzler matrices. Then the set $\spectra(Q^{(1)}, \dots, Q^{(n)})$ consists of the points $x \in \trop^{n}$ such that:
\begin{compactitem}
\item for all $i \in [m]$, $Q_{ii}^{+}(x) \ge Q_{ii}^{-}(x)$;
\item for all $i,j \in [m]$, $i <j$, $Q_{ii}^{+}(x) \tdot Q_{jj}^{+}(x) \ge (Q_{ij}(x))^{\tdot 2}$.
\end{compactitem}
\end{lemma}
Observe that in the lemma above, the term $Q_{ij}(x)$ ($i \neq j$) is well defined for any $x \in \trop^n$ thanks to the Metzler property of the matrices $Q^{(k)}$. We can equivalently rewrite the constraints defining $\spectra$ using classical notation as follows: for all $i \in [m]$, 
\begin{equation}
\max_{Q^{(k)}_{i i} \in \postrop} \bigl(Q^{(k)}_{i i} + x_k \bigr)
\geq 
\max_{Q^{(l)}_{i i} \in \negtrop} \bigl(\abs{Q^{(l)}_{i i}} + x_l \bigr)\, ,\label{eq:first_kind}
\end{equation}
and for all $i, j \in [m]$ such that $i < j$,
\begin{equation}
\begin{multlined}
\max_{Q^{(k)}_{i i} \in \postrop} \bigl(Q^{(k)}_{i i} + x_k\bigr) + \max_{Q^{(k')}_{j j} \in \postrop} \bigl(Q^{(k')}_{j j} + x_{k'} \bigr)
\geq 2 \max_{l \in [n]} \bigl(\abs{Q^{(l)}_{i j}} + x_l \bigr)\, .
\end{multlined} \label{eq:second_kind}
\end{equation}

When the matrices $Q^{(k)}$ are Metzler, we refer to the set $\spectra$ as a \emph{tropical Metzler spectrahedron}. This terminology relies on the fact that in this case we can build a spectrahedron in $\nnpuiseux^n$ which is a lift of $\spectra$~\cite[Proposition~23]{tropical_spectrahedra}. We point out that if we restrict our considerations to the case of Metzler matrices, then the construction of the set $Y$ from Theorem~\ref{theorem:generic_metzler} can be simplified, but it still requires an exponential number of hyperplanes.

\begin{example}
Let us illustrate these results on the spectrahedron defined in Section~\ref{section:illustration}. The corresponding tropical matrices $Q^{(1)}, Q^{(2)}, Q^{(3)} \in \trop^{3 \times 3}$ are of Metzler type, and the associated tropical Metzler spectrahedron $\spectra$ is defined by the constraints:
\begin{align*}
\max\bigl(-1 + x_1, -5/4+x_3\bigr) & \geq x_2 \\
\max\bigl(x_1 + x_3, -1/4 + 2 x_3\bigr) & \geq 2 x_1  \\
x_2 & \geq -7/4 + x_3 \\
\max\bigl(5/4 + x_1 + x_2, 1 + x_2 + x_3\bigr) & \geq 2 x_3 
\end{align*}
The first inequality comes from~\eqref{eq:first_kind} with $i = 2$, and the last three constraints from~\eqref{eq:second_kind}.\footnote{The constraints of the form~\eqref{eq:first_kind} with $i = 1, 3$ are trivial.} The intersection of $\spectra$ with the hyperplane $x_3 = 0$ is depicted in Figure~\ref{fig:tropical_spectrahedron}. The matrices $Q^{(k)}$ fulfill the genericity conditions mentioned in Theorem~\ref{theorem:generic_metzler} so that the tropicalization of the spectrahedron of Section~\ref{section:illustration} is precisely described by the four inequalities given above.
\end{example}

\begin{figure}
\begin{center}
\begin{tikzpicture}[scale=6,convex/.style={draw=lightgray,fill=lightgray,fill opacity=0.7},convexborder/.style={very thick}]
\draw[gray!30,very thin,step=0.125] (-0.5,-1.25) grid (0.25,-1);
\foreach \x in {-0.5,-0.25,0,0.25}
 \node[anchor=north,gray!80,font=\scriptsize] at (\x,-1.25) {\x};
\foreach \y in {-1.25,-1}
 \node[anchor=east,gray!80,font=\scriptsize] at (-0.5,\y) {\y};
\draw[gray!50,->] (-0.5,-1.25) -- (0.325,-1.25) node[color=gray!50,right] {$x_1$};
\draw[gray!50,->] (-0.5,-1.25) -- (-0.5,-0.925) node[color=gray!50,above] {$x_2$};
\filldraw[convex] (-1/8,-9/8) -- (0,-5/4) -- (0,-1) -- cycle;
\draw[convexborder] (-1/8,-9/8) -- (0,-5/4) -- (0,-1) -- cycle;
\end{tikzpicture}
\end{center}
\caption{The tropicalization of the spectrahedron of Section~\ref{section:illustration}.}\label{fig:tropical_spectrahedron}
\end{figure}
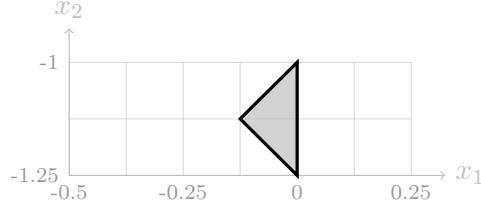

Even though the Metzler case may look special, we can reduce the problem of deciding whether the set $\spectra(Q^{(1)}, \dots, Q^{(n)})$ is trivial to the subproblem in which the matrices are Metzler. To show this, we prove that every set  $\spectra(Q^{(1)}, \dots, Q^{(n)})$ is a projection of a tropical Metzler spectrahedron. Furthermore, this Metzler spectrahedron can be constructed in poly-time.
\begin{proposition}\label{prop:generic_lift}
Let $Q^{(1)}, \dots, Q^{(n)} \in \strop^{m \times m}$ be symmetric tropical matrices. Then the set $\spectra(Q^{(1)}, \dots, Q^{(n)})$ is a projection of a tropical Metzler spectrahedron.
\end{proposition}
\begin{proof}
Let us denote $\spectra(Q^{(1)}, \dots, Q^{(n)})$ by $\spectra$. For every pair $(i,j) \in [m] \times [m]$, $i < j$ we introduce a variable $y_{ij}$ and we consider the set $\tilde{\spectra}$ defined as the set of the points $(x,y) \in \trop^{n} \times \trop^{m(m-1)/2}$ that fulfill the following conditions:
\begin{compactitem}
\item for all $i \in [m]$, $Q_{ii}^{+}(x) \ge Q_{ii}^{-}(x)$;
\item for all $i,j \in [m]$, $i <j$, $y_{ij} \tplus Q_{ij}^{+}(x) \ge Q_{ij}^{-}(x)$;
\item for all $i,j \in [m]$, $i <j$, $y_{ij} \tplus Q_{ij}^{-}(x) \ge Q_{ij}^{+}(x)$;
\item for all $i,j \in [m]$, $i <j$, $Q_{ii}^{+}(x) \tdot Q_{jj}^{+}(x) \ge y_{ij}^{\tdot 2}$.
\end{compactitem}
The set $\tilde{\spectra}$ is a tropical Metzler spectrahedron defined by matrices of size $m^{2} \times m^{2}$. We claim that $\spectra$ is a projection of $\tilde{\spectra}$. First, let us take a point $x \in \spectra$. For every $(i,j)$ with $i < j$ such that $Q_{ij}^{-}(x) = Q_{ij}^{+}(x)$ we put $y_{ij} = \zero$. For every $(i,j)$ with $i < j$ such that $Q_{ii}^{+}(x) \tdot Q_{jj}^{+}(x) \ge (Q_{ij}^{+}(x) \tplus Q_{ij}^{-}(x))^{\tdot 2}$ we put $y_{ij} = Q_{ij}^{+}(x) \tplus Q_{ij}^{-}(x)$. It is clear that we have $(x,y) \in \tilde{\spectra}$. Conversely, let $(x,y) \in \tilde{\spectra}$. For every $(i,j)$ with $i < j$ we consider two cases. If $y_{ij} \ge Q_{ij}^{-}(x)$, then we have $y_{ij} \ge Q_{ij}^{+}(x)$ and hence $Q_{ii}^{+}(x) \tdot Q_{jj}^{+}(x) \ge y_{ij} \ge (Q_{ij}^{+}(x) \tplus Q_{ij}^{-}(x))^{\tdot 2}$. If $y_{ij} < Q_{ij}^{-}(x)$, then we have $Q_{ij}^{+}(x) \ge Q_{ij}^{-}(x)$ and $Q_{ij}^{-}(x) \ge Q_{ij}^{+}(x)$. Hence $Q_{ij}^{+}(x) = Q_{ij}^{-}(x)$. Therefore $x \in \spectra$.
\end{proof}

In the light of Proposition~\ref{prop:generic_lift}, we restrict in the rest of the paper to the problem of deciding whether a tropical Metzler spectrahedron is trivial. 

\section{Tropical spectrahedra and stochastic games}\label{sec:spectra_and_games}

\subsection{Stochastic mean payoff games}\label{section:stochastic_games}

In this section, we present the class of games which is related to nonarchimedean semidefinite feasibility problems. For simplicity, we refer to them as stochastic mean payoff games, although as we shall see, this terminology usually corresponds to a larger class of games. This abuse of language is justified by the fact that the associated decision and computational problems are poly-time equivalent, as discussed below. 

In our setting, a \emph{stochastic mean payoff game} involves two players, Max and Min, who control disjoint sets of states. 
The states owned by Max and Min are respectively indexed by elements of $[m]$ and $[n]$. We will use the symbols $i, j$ to refer to states of Player Max, and $k, l$ to states of Player Min. Both players alternatively move a pawn over these states as follows. When the pawn is on a state $k \in [n]$, Player Min chooses an action $a \in A^{(k)}$, which is defined as a subset of states of Max of cardinality $1$ or $2$: if $a = \{i\}$, then the pawn is moved to the state $i$, while if $a = \{i, j\}$ with $i \neq j$, it is moved to the state $i$ (resp.\ $j$) with probability $1/2$. In both cases, Player Max receives from Player Min a reward denoted by $r^a_k$. Once the pawn is on a state $i \in [m]$, Player Max picks an action $b \in B^{(i)}$, where $b$ is a subset of states of Min of cardinality $1$. Then, Player Max moves the pawn to the state $l$ such that $b = \{l\}$, and Player Min pays him a payment denoted by $r^b_i$. 

We suppose that Player Min starts the game and that players can always make the next move, i.e., that $A^{(k)} \neq \emptyset$ and $B^{(i)} \neq \emptyset$ for all $i \in [m]$ and $k \in [n]$.  

A \emph{policy for Player Min} is a function mapping every state $k \in [n]$ to an action $\sigma(k)$ in $A^{(k)}$. Analogously, a \emph{policy for Player Max} is a function $\tau$ such that $\tau(i) \in B^{(i)}$ for all $i \in [m]$. Suppose that the game starts from a state $k^*$ of Player Min. When players play according to a couple $(\sigma, \tau)$ of policies, the movement of the pawn is described by a Markov chain on the space $\chainspace = [m] \dunion [n]$. The average payoff of Player Max in the long-term is then defined as the average payoff of the controller in this Markov chain. In other words, the payoff of Player Max is given by
\begin{equation}\label{eq:average_payoff}
g_{k^*}(\sigma, \tau) = \lim_{N \to \infty} \E_{\sigma, \tau}\Bigl(\frac{1}{2N} \sum_{p = 1}^N \bigl(r^{\sigma(k_p)}_{k_p} + r^{\tau(i_p)}_{i_p}\bigr)\Bigr) \ , 
\end{equation}
where the expectation $\E_{\sigma, \tau}$ is taken over all the trajectories $k_1, i_1, k_2, \dots, i_p$ starting from $k_1 = k^*$ in the Markov chain. The goal of Player Max is to find a policy which maximizes his average payoff, while Player Min aims at minimizing this quantity. 
The basic theorem of stochastic mean payoff games is the existence of ``optimal'' policies for both players. This was proven by Liggett and Lippman~\cite{liggett_lippman}.
\begin{theorem}\label{th:optimal_policies}
There exists a pair of policies $(\overbar{\sigma}, \overbar{\tau})$ and a unique vector $\gameval \in \R^n$ such that for all initial states $k \in [n]$, the following two conditions are satisfied:
\begin{compactitem}
\item for each policy $\sigma$ of Player Min, $\gameval_k \leq g_k(\sigma, \overbar{\tau})$;
\item for each policy $\tau$ of Player Max, $\gameval_k \geq g_k(\overbar{\sigma}, \tau)$.
\end{compactitem}
\end{theorem}
The vector $\gameval$ in Theorem~\ref{th:optimal_policies} is referred to as the \emph{value} of the game. Note that for every $k$, the quantity $\gameval_k$  coincides with average payoff $g_k(\overbar{\sigma}, \overbar{\tau})$ associated with the couple of policies $(\overbar{\sigma}, \overbar{\tau})$.
The first condition in Theorem~\ref{th:optimal_policies} states that, by playing according to the policy $\overbar{\tau}$, Player Max is certain to get an average payoff greater than or equal to the value $\gameval_k$ associated with the initial state. Symmetrically, Player Min is ensured to limit her average loss to the quantity $\gameval_k$ by following the policy $\overbar{\sigma}$.

A state $k \in [n]$ is said to be \emph{winning (for Player Max)} if the value $\gameval_k$ of the game starting from the initial state $k$ is nonnegative. We denote by \Smpg{} the following decision problem: ``given a stochastic mean payoff game, does there exist an initial state which is winning for Player Max?'' It can be shown that if one of the players has no choice (i.e., has only one possible policy), then the problem of finding values and optimal policies can be solved in polynomial time by linear programming (see, e.g., \cite[Section~ 2.9]{filar_vrieze}). This readily implies that \Smpg{} is in $\NP \cap \coNP$ (one guesses an optimal policy for one player, fixes it and then finds an optimal policy and the value of a $1$-player game). However, as discussed in the introduction, the question whether there exists a poly-time algorithm to solve this problem is open.
\begin{remark}\label{remark:general_games}
In the literature, stochastic mean payoff games correspond to a larger class of problems~\cite{andersson_miltersen}. 
These more general games admit a value and optimal policies as well. It turns out that the associated decision problem is poly-time equivalent to the simpler problem \Smpg{} defined above. Moreover, these decision problems are poly-time equivalent to the problem of computing the value and a pair of optimal policies. We refer to Section~\ref{sec:game_equivalence} for details.
\end{remark}

\begin{example}\label{ex:payoff_computation}
Let us revisit the example presented in Section~\ref{section:illustration}, see Figure~\ref{fig:example_of_a_game}. As noted previously, both players in this example game have only two policies: at state \maxstate{2} Player Max can choose the action that goes to $\minstate{1}$ or the action that goes to $\minstate{3}$, whereas at state \minstate{3} Player Min can choose the action that goes to $\bigl\{ \maxstate{1}, \maxstate{3} \bigr\}$ or the action that goes to $\bigl\{ \maxstate{2}, \maxstate{3} \bigr\}$.
Suppose that Player Max chooses the action that goes to $\minstate{1}$ and that Player Min chooses the action that goes to $\bigl\{ \maxstate{1}, \maxstate{3} \bigr\}$. 
The Markov chain obtained in this way has the transition matrix of form
\[
P =
\begin{bmatrix}
0 & U\\
V & 0
\end{bmatrix} \, ,
\]
where $U$ describes the probabilities of transition from circle states to square states, and $V$ describes the probabilities of transition from square states to circle states, i.e.,
\[
U =
\kbordermatrix{
&    \maxstatesmall{1} & \maxstatesmall{2} & \maxstatesmall{3}\\
\minstatesmall{1} & 1/2 & 1/2 & 0\\
\minstatesmall{2} & 0 & 1 & 0\\
\minstatesmall{3} & 1/2 & 0 & 1/2
} \, , \quad 
V =
\kbordermatrix{
&    \minstatesmall{1} & \minstatesmall{2} & \minstatesmall{3}\\
\maxstatesmall{1} & 0 & 0 & 1\\
\maxstatesmall{2} & 1 & 0 & 0\\
\maxstatesmall{3} & 0 & 1 & 0
} \, .
\]

This chain has only one recurrent class and all states belong to this class. Moreover, it is easy to verify that
\[
\pi = \frac{1}{10}(2, 1, 2, 2, 2, 1)
\]
is the stationary distribution of this chain. Therefore, by Theorem~\ref{theorem:characterization_of_payoff}, the payoff of Player Max is equal to
\[
g\Bigl( \bigl\{ \maxstate{1}, \maxstate{3} \bigr\}, \{ \minstate{1} \} \Bigr) = \frac{1}{10}(-\frac{6}{4} + 2 - 2 + \frac{9}{4}) = \frac{3}{40}. 
\]

Similarly, if Player Max chooses the action that goes to $\minstate{1}$ and Player Min chooses the action that goes to $\bigl\{ \maxstate{2}, \maxstate{3} \bigr\}$, then the transition matrix of the generated Markov chain has the form $
P' =
\begin{bsmallmatrix}
0 & U'\\
V & 0
\end{bsmallmatrix}$, 
where
\[
U' =
\kbordermatrix{
&    \maxstatesmall{1} & \maxstatesmall{2} & \maxstatesmall{3}\\
\minstatesmall{1} & 1/2 & 1/2 & 0\\
\minstatesmall{2} & 0 & 1 & 0\\
\minstatesmall{3} & 0 & 1/2 & 1/2
}
\]
and $V$ is the same as previously. In this case the chain also has only one recurrent class and every state belongs to this class. Furthermore,
\[
\pi' =  \frac{1}{14}(4, 1, 2, 2, 4, 1)
\]
is the stationary distribution of this chain. Hence the payoff of Player Max is equal to
\[
g\Bigl( \bigl\{ \maxstate{2}, \maxstate{3} \bigr\}, \{ \minstate{1}\} \Bigr)  = \frac{1}{14}(2 - 4 + \frac{9}{4}) = \frac{1}{56}.
\]In both cases, the payoff of Player Max is positive. Therefore, if Player Max chooses the action that goes to $\minstate{1}$, then he is guaranteed to obtain a nonnegative payoff. One can check, by doing the same calculations for the remaining policies, that $\Bigl( \bigl\{ \maxstate{2}, \maxstate{3} \bigr\}, \{ \minstate{1}\} \Bigr)$ is the unique couple of optimal policies in this game.
\end{example}

\subsection{Shapley operators}

One of the possible approaches to analyze stochastic mean payoff games is to introduce the associated \emph{Shapley operator}, which is a map $\shapley$ from $\trop^{n}$ to itself defined as:
\begin{align}\label{e-def-shapley}
(\shapley(x))_k =
\min_{
\substack{a \in A^{(k)}\\
a = \{i, j\}}}
\Bigl(r^a_k + \frac{1}{2}\bigl(\max_{
\substack{
b \in B^{(i)}\\
b = \{l\}}} (r^b_i + x_l) 
+ \max_{
\substack{
b \in B^{(j)}\\
b = \{l\}}} (r^b_j + x_l)\bigr)\Bigr) \, ,
\end{align}
where we use the convention that $i = j$ when $a \in A^{(i)}$ is the singleton $\{i\}$ and we set $\frac{1}{2}(-\infty) = -\infty$.

Let us point out that $\trop$ can be endowed with a natural metric $d(x,y) = \abs{\exp(x) - \exp(y)}$. This metric generates the product topology on $\trop^{n}$ in which the functions
\begin{equation*}
\begin{aligned}
&\trop^{2} \ni (a,b)  \to a + b \in \trop, \
\trop^{2} \ni (a,b)  \to \max\{a, b\} \in \trop, \\
&\trop^{2} \ni (a,b)  \to \min\{a, b\} \in \trop, \
\trop \ni a  \to \frac{1}{2}a \in \trop
\end{aligned}
\end{equation*}
are continuous. If $x, y \in \trop^{n}$ are vectors, then we denote $x \le y$ if the inequality $x_{k} \le y_{k}$ is fulfilled for every $k \in [n]$. We say that a map $f \colon \R^{n} \to \R^{d}$ is \emph{piecewise affine} if $\R^{n}$ can be partitioned into finitely many polyhedra such that $f$ restricted to any of these polyhedra is affine. The basic properties of Shapley operators are summarized in the next lemma. Hereafter, if $\lambda \in \R$ and $x \in \trop^{n}$, then we use the notation $\lambda + x$ to denote the vector $(\lambda + x_{1}, \dots, \lambda + x_{n})$.

\begin{lemma}\label{lemma:basic_shapley_properties} The Shapley operator $\shapley$ has the following properties:
\begin{enumerate}[(i)]
\item it is \emph{order preserving}, i.e., for all $x, y \in \trop^{n}$ such that $x \le y$ we have $\shapley(x) \le \shapley(y)$;
\item it is \emph{additively homogeneous}, i.e., if and $\lambda \in \R$, then $\shapley(\lambda + x) = \lambda + \shapley(x)$ for all $x \in \trop^{n}$;
\item it is continuous in the topology of $\trop^{n}$;
\item\label{point:nonexpansive} $\shapley_{|\R^{n}}$ is \emph{nonexpansive} in the supremum norm, i.e., for all $x, y \in \R^{n}$ we have $\|\shapley(x) -  \shapley(y)\| \le \| x - y\|$, where $\|x\| = \max_{k \in [n]}|x_{k}|$;
\item $\shapley_{|\R^{n}}$ is piecewise affine.
\end{enumerate}
\end{lemma}
\begin{proof}
The first two properties follow trivially from the definition of $\shapley$. The third follows from the remarks about topology of $\trop^{n}$ stated above. The fourth follows from the first two and the fact that $\shapley$ preserves $\R^{n}$. Indeed, for all $x, y \in\R^{n}$ we have $x \le \| x - y\| + y$. Therefore $\shapley(x) \le \shapley(\| x - y\| + y) = \| x - y\| + \shapley(y)$. Analogously, $\shapley(y) \le \| x - y\| + \shapley(x)$ and hence $\|\shapley(x) -  \shapley(y)\| \le \| x - y\|$. The last point is immediate from the definition of $\shapley$.
\end{proof}

A central question, given $f \colon \R^{n} \to \R^{n}$, is to decide whether the limit $\lim_{N \to \infty}f^{N}(x)/N$ exists. In order to prove the existence of this limit, 
Kohlberg 
introduced the notion of an invariant half-line and proved the next theorem~\cite{kohlberg}.

\begin{theorem}\label{theorem:kohlberg}
Suppose that function $f \colon \R^{n} \to \R^{n}$ is piecewise affine and nonexpansive in any norm. Then there exists a triple $(\bias, \eigenval, \gamma_{0}) \in \R^{n} \times \R^{n} \times \R$, $\gamma_{0} \ge 0$, such that $f(\bias + \gamma \eigenval) = \bias + (\gamma+1)\eigenval$ for all $\gamma \ge \gamma_{0}$. A pair $(\bias, \eigenval)$ in called an \emph{invariant half-line}. Furthermore, if $(\bias_{1}, \eigenval_{1})$ and $(\bias_{2}, \eigenval_{2})$ are invariant half-lines, then $\eigenval_{1} = \eigenval_{2}$.
\end{theorem}

In particular, we obtain $\lim_{N \to \infty} f^{N}(\bias)/N = \eigenval$. Moreover, for any $x \in \R^{n}$ and $N \ge 1$ we have $\| f^{N}(x) - f^{N}(\bias)\| \le \| x - \bias\|$. Thus, we get the following corollary.

\begin{corollary}\label{corollary:limit_of_shapley}
Suppose that $f \colon \R^{n} \to \R^{n}$ fulfills the assumptions of Theorem~\ref{theorem:kohlberg}. Then, for every $x \in \R^{n}$ we have
\[
\lim_{N \to \infty} \frac{1}{N}f^{N}(x) = \eigenval \ \text{ and } \ \lim_{N \to \infty} \max_{k \in [n]} \frac{1}{N}f^{N}_{k}(x) = \max_{k \in [n]} \eigenval_{k} \, .
\]
\end{corollary}

Lemma~\ref{lemma:basic_shapley_properties} and Theorem~\ref{theorem:kohlberg} show that every Shapley operator $\shapley$ has an invariant half-line $(\bias, \eigenval)$. The vectors $x$ such that $x \leq \shapley(x)$ may be thought of as nonlinear
analogues of subharmonic functions. 
Akian, Gaubert, and Guterman showed
that the existence of such a nontrivial vector
$x$
is equivalent to the
property that the mean payoff game has at least one winning state, 
in the case of deterministic games~\cite{polyhedra_equiv_mean_payoff}. This
is based in particular on a nonlinear fixed point theorem (Collatz--Wielandt theorem), building on earlier work of Nussabum~\cite[Theorem~3.1]{nussbaum}.

\begin{theorem}[{Collatz--Wielandt properties, \cite[Lemma~2.8]{polyhedra_equiv_mean_payoff}}]\label{theorem:collatz-wielandt}
Let $f \colon \trop^{n} \to \trop^{n}$ be an order-preserving, additively homogeneous, and continuous map. Let $x \in \R^{n}$ be any point. Then
\begin{align} 
\lim_{N \to \infty}\max_{k \in [n]} \frac{1}{N}f^{N}_{k}(x) &=\max \{ \pert \in \trop \colon \exists u \in \trop^{n}, u \neq \zero, f(u) \ge \pert + u\} \, ,\label{pty1:collatz-wielandt}\\
& = \inf \{ \pert \in \R\colon \exists u\in \R^n,\; f(u)\leq \pert +u\}
\label{pty2:collatz-wielandt}
\end{align}
where we use the notation $\trop^{n} \ni \zero = (\zero, \zero, \dots, \zero)$.
\end{theorem}
Note that this theorem is stated in
\cite{polyhedra_equiv_mean_payoff} 
with ``$\sup$'' instead of ``$\max$'' on the right hand side of~\eqref{pty1:collatz-wielandt}, but their proof shows that this supremum is attained. The infimum in~\eqref{pty2:collatz-wielandt} is not attained in general.

The limit $\eigenval$ from Corollary~\ref{corollary:limit_of_shapley} is closely related to the value  $\gameval$ of the corresponding mean payoff game. More precisely, we have $\eigenval = 2\gameval$. In order to prove this statement, we need the following notation: for every fixed policy $\sigma$ of Player Min, we denote by $\shapley^{\sigma}$ the Shapley operator of a $1$-player game in which Player Min can only use $\sigma$. Analogously, for every fixed policy $\tau$ of Player Max, we denote by $\shapley^{\tau}$ the Shapley operator of a $1$-player game in which Player Min can only use $\tau$. By $\shapley^{\sigma, \tau}$ we denote the Shapley operator of a $0$-player game in which Player Min can only use $\sigma$ and Player Max can only use $\tau$. In other words, these operators are defined as follows:
\begin{align*}
\shapley^{\sigma}(x)_{k} &= r^{\sigma(k)}_k + \frac{1}{|\sigma(k)|}\sum_{j \in \sigma(k)}\max_{
\substack{
b \in B^{(j)}\\
b = \{l\}}} (r^b_{j} + x_{l}) \, ,\\
\shapley^{\tau}(x)_{k} &= \min_{a \in A^{(k)}}\Bigl(r^{a}_{k} + \frac{1}{|a|}\sum_{
\substack{
j \in a\\
\tau(j) = \{l\}
}} (r^{\tau(j)}_{j} + x_{l}) \Bigr) \, ,\\
\shapley^{\sigma, \tau}(x)_{k} &= r^{\sigma(k)}_k + \frac{1}{|\sigma(k)|}\sum_{
\substack{
j \in \sigma(k)\\
\tau(j) = \{l\}
}} (r^{\tau(j)}_{j} + x_{l}) \, .
\end{align*}

It is easy to verify that we have the following selection lemma.
\begin{lemma}\label{lemma:two_players_from_one_player} For every $x \in \R^{n}$ we have
\[
\shapley(x) = \min_{\sigma}\shapley^{\sigma}(x) = \max_{\tau}\shapley^{\tau}(x) \, ,
\]
where $\min$ and $\max$ denote the minimum and maximum in the partial order $\le$ on $\R^{n}$, the minimum is taken over the set of policies of Player Min, and the maximum is taken over the set of policies of Player Max.  Similarly, $\shapley^{\sigma}(x) =\max_{\tau}\shapley^{\sigma, \tau}(x)$ and $\shapley^{\tau}(x) = \min_{\sigma}\shapley^{\sigma, \tau}(x)$.
\end{lemma}
Moreover, an easy induction shows the following result:

\begin{lemma}\label{lemma:finite_game}
Suppose that Player Min uses a policy $\sigma$ and that Player Max uses a policy~$\tau$. Then, for every initial state $k = k_{1} \in [n]$ the expected total payoff of Player Max after the $N$th turn is equal to $\bigl((\shapley^{\sigma,\tau})^{N}(0)\bigr)_{k}$. In other words, we have
\[
\forall N, \ \E_{\sigma, \tau}\Bigl(\sum_{p = 1}^N \bigl(r^{\sigma(k_p)}_{k_p} + r^{\tau(i_p)}_{i_p}\bigr)\Bigr) = \bigl( (\shapley^{\sigma, \tau} \circ \shapley^{\sigma, \tau} \circ \dots \circ \shapley^{\sigma, \tau}) (0) \bigr)_{k} = \bigl( (\shapley^{\sigma,\tau})^{N}(0) \bigr)_{k} \, .
\]
\end{lemma}
\begin{proof}
Let us denote the expected total payoff of Player Max after the $N$th turn by $f(N, k)$, where $k \in [n]$ denotes the initial state. Observe that we have the relations
\[
f(1, k) = r^{\sigma(k)}_k + \frac{1}{|\sigma(k)|}\sum_{
\substack{
j \in \sigma(k)\\
\tau(j) = \{l\}
}} r^{\tau(j)}_{j} = \shapley^{\sigma, \tau}(0)_{k}
\]
and 
\[
f(N + 1, k) = r^{\sigma(k)}_k + \frac{1}{|\sigma(k)|}\sum_{
\substack{
j \in \sigma(k)\\
\tau(j) = \{l\}
}} \bigl(r^{\tau(j)}_{j} + f(N,l) \bigr) \, .
\]
This proves by induction that $f(N, k) = \bigl( (\shapley^{\sigma,\tau})^{N}(0) \bigr)_{k}$.
\end{proof}

\begin{lemma}\label{lemma:same_invariant_line}
There exists a pair of policies $(\overbar{\sigma}, \overbar{\tau})$ such that $\shapley$, $\shapley^{\overbar{\sigma}}$, and $\shapley^{\overbar{\tau}}$ have the same invariant half-line $(\bias, \eigenval)$.
\end{lemma}
\begin{proof}
Let $(\bias, \eigenval)$ denote any invariant half-line of $\shapley$. We start by showing the existence of~$\overbar{\sigma}$. By Lemma~\ref{lemma:two_players_from_one_player} we have $\shapley(\bias + \gamma \eigenval) = \min_{\sigma} \shapley^{\sigma}(\bias + \gamma \eigenval)$. Therefore, there exists a strategy $\overbar{\sigma}$ such that the equality $\bias + (N + 1)\eigenval = \shapley^{\overbar{\sigma}}(\bias + N \eigenval)$ is verified for infinitely many values of $N \in \N^{*}$. Since the operator $\shapley^{\overbar{\sigma}}$ is piecewise-affine, there exists an affine function $Ax + b$ such that $\shapley^{\overbar{\sigma}}(\bias + \gamma \eigenval) = A(\bias + \gamma \eigenval) + b$ for all $\gamma \ge 0$ large enough. In particular, $\bias + \eigenval + N\eigenval = A\bias + b + N A\eigenval$ for infinitely many values of $N \in \N^{*}$. Hence $A\eigenval = \eigenval$ and $\bias + \eigenval = A \bias + b$. Therefore $\shapley^{\overbar{\sigma}}(\bias + \gamma \eigenval) = A(\bias + \gamma \eigenval) + b = \bias + (\gamma + 1)\eigenval$ for all $\gamma \ge 0$ large enough. The construction of $\overbar{\tau}$ is analogous.
\end{proof}

\begin{theorem}\label{theorem:optimal_policies_proof}
Policies $(\overbar{\sigma}, \overbar{\tau})$ are optimal. Furthermore, the value of the game is equal to $\eigenval/2$.
\end{theorem}
\begin{proof}
Suppose that Player Min uses $\overbar{\sigma}$. Let $\tau$ be any policy of Player Min. By Lemma~\ref{lemma:finite_game} we have
\[
\forall \tau, \ g_k(\overbar{\sigma}, \tau) = \lim_{N \to \infty} \E_{\overbar{\sigma}, \tau}\Bigl(\frac{1}{2N} \sum_{p = 1}^N \bigl(r^{\overbar{\sigma}(k_p)}_{k_p} + r^{\tau(i_p)}_{i_p}\bigr)\Bigr) = \lim_{N \to \infty}\frac{(\shapley^{\overbar{\sigma}, \tau})^N(0)}{2N} \, .
\]
Therefore, Lemma~\ref{lemma:two_players_from_one_player}, Lemma~\ref{lemma:same_invariant_line}, and Corollary~\ref{corollary:limit_of_shapley} show that
\[
\forall \tau, \ g_k(\overbar{\sigma}, \tau) \le \lim_{N \to \infty}\frac{(\shapley^{\overbar{\sigma}})^N(0)}{2N} = \frac{1}{2}\eigenval_{k}.
\]
Analogously, if Player Max uses $\overbar{\tau}$, we have
\[ 
\forall \sigma, \ g_k(\sigma, \overbar{\tau}) = \lim_{N \to \infty} \E_{\sigma, \overbar{\tau}}\Bigl(\frac{1}{2N} \sum_{p = 1}^N \bigl(r^{\sigma(k_p)}_{k_p} + r^{\overbar{\tau}(i_p)}_{i_p}\bigr)\Bigr) \ge \frac{1}{2}\eigenval_{k} \, .
\]
Hence $(\overbar{\sigma}, \overbar{\tau})$ are optimal and $g_{k}(\overbar{\sigma}, \overbar{\tau}) = \eigenval_{k}/2$ for every $k \in [n]$.
\end{proof}

\begin{theorem}\label{theorem:shapley_and_games}
Player Max has at least one winning initial state if and only if the set $\{x \in \trop ^{n} \colon x \le \shapley(x) \}$ is nontrivial (i.e., contains a point different than $(\zero, \dots, \zero)$).
\end{theorem}
\begin{proof}
By Theorem~\ref{theorem:optimal_policies_proof} and Corollary~\ref{corollary:limit_of_shapley} we have
\[
\max_{k \in [n]} 2\gameval_{k} = \lim_{N \to \infty} \max_{k \in [n]} \frac{1}{N}\shapley^{N}_{k}(x) \, .
\]
Hence, by~\eqref{pty1:collatz-wielandt}, we have
\[
\max_{k \in [n]} 2\gameval_{k} = \max \{ \pert \in \trop \colon \exists x \in \trop^{n}, x \neq \zero, \pert + x \le \shapley(x) \} \, .
\]
In particular, the inequality $\max_{k \in [n]} \gameval_{k} \ge 0$ holds if and only if the set $\{x \in \trop ^{n} \colon x \le \shapley(x) \}$ is nontrivial.
\end{proof}

The Shapley operator $F$ may be thought of as a non-linear Markov operator,
and so, the vectors $x$ such that $x\leq F(x)$ may be thought of as
non-linear {\em sub-harmonic} vectors. Theorem~\ref{theorem:shapley_and_games}
shows that the existence of a winning state is characterized by 
the existence of a non-trivial sub-harmonic vector.

\section{Equivalence between stochastic games and tropical spectrahedra}\label{sec:tropical_games}

\subsection{The stochastic game associated with a Metzler tropical spectrahedron}
We now describe the connection between tropical spectrahedra and stochastic mean payoff games. Let us consider a tropical Metzler spectrahedron $\spectra$ associated with the matrices $Q^{(1)}, \dots, Q^{(n)} \in \strop^{m \times m}$. We construct a stochastic mean payoff game $\Gamma$ consisting of $m$ states of Player Max and $n$ states of Player Min. For each state $k \in [n]$ of Player Min, the set $A^{(k)}$ of actions available to Player Min at this state consists of:
\begin{compactitem}
\item the actions $\{i\}$ with payment $-\abs{Q^{(k)}_{i i}}$ for all $i \in [m]$ such that $Q^{(k)}_{i i} \in \negtrop$;
\item the actions $\{i, j\}$ with payment $-\abs{Q^{(k)}_{i j}}$ for all $i < j$ such that $Q^{(k)}_{i j} \in \negtrop$.
\end{compactitem}
For every state $i \in [m]$ of Player Max, the set $B^{(i)}$ of actions available to Player Max at this state is formed by the actions $\{k\}$ with payment $Q^{(k)}_{i i}$ for all states $k \in [n]$ satisfying $Q^{(k)}_{i i} \in \postrop$.

Recall that in the games which we consider, every state has to be equipped with at least one action, \ie, the sets $A^{(k)}$ and $B^{(i)}$ must be nonempty. In consequence, our construction is valid provided that the following assumption on the matrices $Q^{(k)}$ is satisfied:
\begin{assumption}\label{assump:well_formed}
\begin{compactenum}[(a)]
\item\label{item:min} For all $k \in [n]$, the matrix $Q^{(k)}$ has at least one coefficient in $\negtrop$.
\item\label{item:max} For all $i \in [m]$, there exists $k \in [n]$ such that the diagonal coefficient $Q^{(k)}_{i i}$ belongs to $\postrop$.
\end{compactenum}
\end{assumption}
Concerning the nontriviality of tropical Metzler spectrahedra, Assumption~\ref{assump:well_formed} can be made without loss of generality, up to extracting submatrices from $Q^{(k)}$ or eliminating some of them, as shown by the next lemma.

\begin{lemma}\label{lemma:well_formed}
Let $Q^{(1)}, \dots, Q^{(n)} \in \strop^{m \times m}$ be symmetric Metzler matrices, and $\spectra$ be the associated tropical spectrahedron. We can build (in poly-time) symmetric Metzler matrices  $R^{(1)}, \dots, R^{(q)} \in \strop^{p \times p}$ ($p \leq m$, $q \leq n$) satisfying Assumption~\ref{assump:well_formed} such that the associated tropical spectrahedron is nontrivial if and only if $\spectra$ is nontrivial.
\end{lemma}

\begin{proof}
We first examine Assumption~\ref{assump:well_formed}\eqref{item:min}. Suppose that the matrix $Q^{(k)}$ has no coefficient in $\negtrop$. 
In this case, the spectrahedron $\spectra$ is nontrivial, since it contains the vector $x \in \trop^n$ such that $x_k = 0$ and $x_l = \zero$ 
for all $l \neq k$. 

Now, let us look at Assumption~\ref{assump:well_formed}\eqref{item:max}. Consider $i \in [m]$ and suppose that no diagonal coefficient $Q^{(k)}_{i i}$ is in $\postrop$. We distinguish three cases:
\begin{itemize}
\item if the set $K$ of indices $k \in [n]$ such that $Q^{(k)}_{i i} \in \negtrop$ is nonempty, the relation $Q^+_{i i}(x) \geq Q^-_{i i}(x)$ enforces to have $x_k = \zero$ for all $x \in \spectra$ and $k \in K$. This means that we can reduce the nontriviality of $\spectra$ to the nontriviality of the spectrahedron associated with the matrices $Q^{(k)}$ with $k \not \in K$;
\item if the aforementioned set $K$ is empty, and the $i$-th row and column of the matrices $Q^{(k)}$ are all identically equal to $\zero$, then we can remove all these rows and columns, and reduce to a problem with matrices of order $m-1$ over the variables $x_1, \dots, x_n$;
\item if the set $K$ is empty and some matrix $Q^{(k)}$ contains an entry different than $\zero$ on its $i$-th row, namely $Q^{(k)}_{i j} \in \negtrop$ with $i < j$, then the relation $Q_{i i}^{+}(x) \tdot Q_{j j}^{+}(x) \geq (Q_{i j}(x))^{\tdot 2}$ enforces $x_{k} = \zero$ for any $x \in \spectra$. We consequently reduce the problem to the nontriviality of the spectrahedron associated with the matrices $Q^{(l)}$ with $l \neq k$. \qedhere
\end{itemize}
\end{proof}

We are going to show that the value of the game $\Gamma$ is related with the feasibility of sets $\pertspectra$ obtained from $\spectra$ by adding a reinforcing coefficient to the inequalities that correspond to minors of order $1$ and $2$.
\begin{definition}\label{def:sublevel_set}
For any $\pert \in \R$ we denote by $\pertspectra$ the set of all points $x \in \trop^{n}$ verifying
\begin{compactitem}
\item for all $i \in [m]$, $Q_{ii}^{+}(x) \ge \lambda \tdot Q_{ii}^{-}(x)$;
\item for all $i,j \in [m]$, $i <j$, $Q_{ii}^{+}(x) \tdot Q_{jj}^{+}(x) \ge (\lambda \tdot Q_{ij}(x))^{\tdot 2}$.
\end{compactitem}
\end{definition}
Observe that we have $\spectra_{0} = \spectra$.

\begin{lemma}\label{lemma:spectrahedron_encodes_shapley}
Let $\shapley \colon \trop^n \to \trop^n$ be the Shapley operator associated with $\Gamma$. Then, for any $\lambda \in \R$ we have $\spectra_{\lambda} = \{x \in \trop^{n} \colon \lambda + x \le \shapley(x) \}$.
\end{lemma}

\begin{proof}
By construction of the game $\Gamma$, a vector $x \in \trop^n$ satisfies $\lambda + x \leq \shapley(x)$ if and only if for all $k \in [n]$,
\[
\lambda + x_k \leq \min_{Q^{(k)}_{i j} \in \negtrop} \Bigl(-\abs{Q^{(k)}_{i j}} + \frac{1}{2} \bigl(\max_{Q^{(l)}_{i i} \in \postrop} (Q^{(l)}_{i i} + x_l)
+ \max_{Q^{(l)}_{j j} \in \postrop} (Q^{(l)}_{j j} + x_l)\bigr) \Bigr) \, ,	
\]
or, equivalently, for all $i, j \in [m]$,
\[
2 \Bigl(\lambda + \max_{\;Q^{(k)}_{i j} \in \negtrop} \bigl(\abs{Q^{(k)}_{i j}} + x_k\bigr) \Bigr) \leq 
\max_{Q^{(l)}_{i i} \in \postrop} \bigl(Q^{(l)}_{i i} + x_l \bigr) \\[-1ex]
+ \max_{Q^{(l)}_{j j} \in \postrop} \bigl(Q^{(l)}_{j j} + x_l \bigr) \, .
\]
By distinguishing whether $i$ and $j$ are equal in these inequalities, and recalling that $Q_{ij}^{(k)}$ is in $\negtrop \cup \{\zero\}$ for all $k \in [n]$ when $i \neq j$, we recover the constraints that describe $\spectra_{\lambda}$.
\end{proof}
From~\eqref{pty1:collatz-wielandt} and Lemma~\ref{lemma:spectrahedron_encodes_shapley}, we obtain:
\begin{theorem}\label{theorem:spectra_and_game}
The set $\spectra_\lambda$ is nontrivial if and only if $\lambda \geq 2 \max_{k \in [n]} \gameval_k$, where $\gameval$ is the value of the game $\Gamma$.

In particular, the tropical spectrahedron $\spectra$ is nontrivial if and only if the stochastic game $\Gamma$ has at least one winning initial state.
\end{theorem}

In case at least one initial state of the game has a positive value, the latter statement can be refined in order to obtain an explicit point in the spectrahedron $\bspectra$.
\begin{lemma}\label{lemma:lift}
Suppose that $x \in \spectra_\lambda$ for some positive $\lambda$. 
Let $\bo Q^{(1)},\dots,\bo Q^{(n)}$ denote any symmetric matrices in $\puiseux^{m\times m}$ such that $\sval (Q^{(k)})=Q^{(k)}$ for all $k\in [n]$.
Then, the point $\bo x = (t^{x_1}, \dots, t^{x_n})$ belongs to $\bspectra\coloneqq\{ \bo x\in \nnpuiseux^n: \bo x_1 \bo Q^{(1)}  + \dots + \bo x_n \bo Q^{(n)}  \loew 0 \}$. 
\end{lemma}

\begin{proof}
The proof of \cite[Lemma~26]{tropical_spectrahedra} shows that if $x \in \spectra_\lambda$ for some positive $\lambda$, then any point $\bo x \in \nnpuiseux^{n}$ verifying $\val(\bo x) = x$ belongs to $\bspectra$.
\end{proof}
\begin{remark}
In case all the initial states of the game have a negative value, we know
from the second Collatz--Wielandt identity~\eqref{pty2:collatz-wielandt} that there exists
a vector $u\in \R^n$ and a scalar $\lambda<0$ such that $F(u)\leq \lambda +u$.
The pair $(u,\lambda)$ yields an emptiness certificate for the nonarchimedean
spectrahedron $\bspectra$.
\end{remark}

Along the same lines, we can reciprocally associate a tropical spectrahedron with any stochastic game. In more details, let $\Gamma$ be a stochastic mean payoff game over the sets of states $[m]$ and $[n]$. We recall that $A^{(k)}$ and $B^{(i)}$ correspond to the sets of possible actions of Player Min at state $k \in [n]$ and Player Max at $i \in [m]$ respectively. Furthermore, $r^{a}_{k}$ is the reward obtained by Player Max when Player Min chooses an action $a \in A^{(k)}$ at state $k \in [n]$ and $r^{b}_{i}$ is the reward obtained by Player Max when he chooses an action $b \in B^{(i)}$ at state $i \in [m]$. Out of this data we construct symmetric Metzler matrices $Q^{(1)}, \dots, Q^{(n)} \in \strop^{m \times m}$. We define $Q^{(k)}_{i j} = Q^{(k)}_{j i} \coloneqq \tminus (-r^a_k)$ for all $k \in [n]$ and $i, j \in [m]$ such that $i \neq j$ and $a = \{i, j\}$ is an available action of Player Min at the state~$k$. Diagonal coefficients of the matrices $Q^{(k)}$ are defined according to whether the actions $a \coloneqq \{i\}$ and $b \coloneqq \{k\}$ are available from the states~$k$ and~$i$ respectively: 
\begin{compactitem}
\item if $a \in A^{(k)}$ and $b \not  \in B^{(i)}$, we set $Q^{(k)}_{i i} \coloneqq \tminus (-r^a_k)$, while if $b\in B^{(i)}$ and $a \not \in A^{(k)}$, we define $Q^{(k)}_{i i} \coloneqq r^b_i$;
\item if both actions $a$ and $b$ occur simultaneously, we set $Q^{(k)}_{i i}$ to $\tminus (-r^a_k)$ if $(-r^a_k) > r^b_i$, and to $r^b_i$ if $r^b_i \geq (-r^a_k)$.
\end{compactitem}
Finally, all the other entries of the matrices $Q^{(k)}$ are set to $\zero$. With this construction, it can be verified that Lemma~\ref{lemma:spectrahedron_encodes_shapley} and, subsequently, Theorem~\ref{theorem:spectra_and_game} are still valid. For the sake of readability, we prove this in Appendix~\ref{app:equivalence}.

We denote by \Tropicalsdfp{} the tropical Metzler semidefinite feasibility problem: ``given symmetric tropical Metzler matrices $Q^{(1)}, \dots, Q^{(n)} \in \strop^{m \times m}$ satisfying Assumption~\ref{assump:well_formed}, is the associated  tropical Metzler spectrahedron trivial?'' As a consequence of Theorem~\ref{theorem:spectra_and_game}, we obtain the equivalence between the two decision problems \Smpg{} and \Tropicalsdfp. This equivalence can be refined thanks to the fact that \Smpg{} restricted to games with payments equal to $0$ or $\pm 1$ is poly-time equivalent to \Smpg{} for arbitrary payments. We show the proof of this reduction in Section~\ref{sec:game_equivalence}. This yields the following result:
\begin{theorem}\label{theorem:simple_games}
The problems \Smpg{} and \Tropicalsdfp{} are poly-time equivalent. Furthermore, if either of these problems can be solved in pseudopolynomial time, then both of them can be solved in polynomial time.
\end{theorem}
\begin{proof}
The fact that \Tropicalsdfp{} is poly-time equivalent to \Smpg{} follows directly from Theorem~\ref{theorem:spectra_and_game}. Moreover, the same statement shows that \Smpg{} restricted to games with payoffs in $\{-1, 0, 1\}$ is poly-time equivalent to \Tropicalsdfp{} restricted to matrices with entries in 
\[ \{0, \tminus 0, \pm 1, \tminus (\pm 1), \zero\} \, .
\]
By using Corollary~\ref{cor:equivalence_pseudopolynomial} (Section~\ref{sec:game_equivalence}), we deduce that if either of these problems can be solved in pseudopolynomial time, then both are solvable in polynomial time.
\end{proof}

We now extend Theorem~\ref{theorem:simple_games} in two different directions. In Section~\ref{sec:nonconic}, we deal with the problem of checking whether a certain stratum of a tropical Metzler spectrahedron is empty. This allows us to handle the case of affine spectrahedra. Section~\ref{sec:archimedean} provides a result on the asymptotic feasibility of real spectrahedra arising from a spectrahedron over Puiseux series.

\subsection{Strata of tropical spectrahedra and dominions}\label{sec:nonconic}

We want to characterize the possible supports of points belonging to $\{x \in \trop^{n} \colon x \le \shapley(x) \}$. These supports can be interpreted in terms of the associated stochastic mean payoff game, using the following notion.
A dominion of Player Max is a subset of initial states $\dominion \subset [n]$ such that, if the game starts in $\dominion$, then Player Max can ensure that the game never reaches any state from $[n] \setminus \dominion$. Formally, we make the following definition.

\begin{definition}
We say that a subset $\dominion \subset [n]$ is a \emph{dominion (of Player Max)} if for every state $k \in \dominion$ and every action $\{i,j\} \in A^{(k)}$, there exists a pair of states $l_{1} \in \dominion$, $l_{2} \in \dominion$ such that $\{l_{1} \} \in B^{(i)}$ and $\{ l_{2} \} \in B^{(j)}$.
\end{definition}

We refer to Figure~\ref{fig:dominion} for an illustration. Given a dominion $\dominion \subset [n]$, we can construct a \emph{subgame induced by $\dominion$}. This game is created as follows: the set of stated controlled by Player Min is equal to $\dominion$. Each of these states is equipped with the same set of actions as in the original game. The set of states controlled by Player Max consists of all states $i \in [m]$ such that there exists a state $k \in \dominion$ and a state $j \in [m]$ verifying $\{i,j\} \in A^{(k)}$. Every such state $i$ is equipped with all the actions of the form $\{l \} \in B^{(i)}$ with $l \in \dominion$. The payoffs associated with the actions in the induced subgame are the same as payoffs in the original game.

\begin{definition}
We say that a dominion $\dominion \subset [n]$ is a \emph{winning dominion} is all states of $\dominion$ are winning in the subgame induced by $\dominion$.
\end{definition}

\begin{figure}[t]
\centering
\begin{tikzpicture}[scale=0.75,>=stealth',max/.style={draw,rectangle,minimum size=0.5cm},min/.style={draw,circle,minimum size=0.5cm},av/.style={draw, circle,fill, inner sep = 0pt,minimum size = 0.2cm}]

\node[min] (min1) at (0,0) {$1$};
\node[min] (min2) at (4,0) {$2$};
\node[min] (min3) at (8,2) {$3$};
\node[min] (min4) at (8,-2) {$4$};

\node[max] (max1) at (2,0) {$1$};
\node[max] (max2) at (6,2) {$2$};
\node[max] (max3) at (6,-2) {$3$};

\node[av] (av23) at (5, 0){};

\draw[->] (min1) to[out=60, in = 120] node[above, font=\small]{$0$} (max1);
\draw[->] (max1) to[out = -120, in = -60] node[below, font=\small]{$-1$} (min1);

\draw[->] (max1) to node[above, font=\small]{$0$} (min2);
\draw[-] (min2) to node{} (av23);
\draw[->] (av23) to node{} (max2);
\draw[->] (av23) to node{} (max3);

\draw[->] (min3) to[out=120, in = 60] node[above, font=\small]{$2$} (max2);
\draw[->] (max2) to[out = -60, in = -120] node[below, font=\small]{$0$} (min3);

\draw[->] (min4) to[out=120, in = 60] node[above, font=\small]{$0$} (max3);
\draw[->] (max3) to[out = -60, in = -120] node[below, font=\small]{$-1$} (min4);

\end{tikzpicture}
\caption[A mean payoff game illustrating the notion of winning dominions.]{A mean payoff games illustrating the notion of winning dominions. The states $\minstate{1}, \minstate{2},$ and $\minstate{3}$ are winning. The minimal (inclusion-wise) dominions are given by $\{ \minstate{1} \}, \{ \minstate{3}\}, \{ \minstate{4}\}$, and $\{\minstate{2}, \minstate{3}, \minstate{4}\}$. The state $\minstate{1}$ is winning and constitutes a dominion, because Player Max can force to return to $\minstate{1}$ by playing $\maxstate{1}\to\minstate{1}$. However, this state does not belong to any winning dominion, because the induced subgame, reduced to the states $\minstate{1}$ and $\maxstate{1}$, is not winning for Player Max. The state $\minstate{3}$ is the only winning dominion.}\label{fig:dominion}
\end{figure}

\begin{remark}
It follows from the definition that if a state $k \in [n]$ belongs to a winning dominion, then it is a winning state. The converse is not true, even if we suppose that $k$ belongs to a dominion that contains only winning states. An example of such situation is presented in Figure~\ref{fig:dominion}.
\end{remark}

We can now give a characterization of all possible supports of points $\{x \in \trop^{n} \colon x \le \shapley(x) \}$.

\begin{theorem}
Take a Shapley operator $\shapley \colon \trop^{n} \to \trop^{n}$ and a nonempty set $K \subset [n]$. Then, the set $\{x \in \trop^{n} \colon x \le \shapley(x) \}$ contains a point with support $K$ if and only if $K$ is a winning dominion in the stochastic mean payoff game associated with $\shapley$.
\end{theorem}
\begin{proof}
Suppose that we are given a point $x \in \trop^{n}$ with support $K$. Then, for every $i \in [m]$ we have
\begin{equation}
\max_{
\substack{
b \in B^{(i)}\\
b = \{l\}}} (r^b_i + x_l) = 
\max_{
\substack{
b \in B^{(i)}\\
b = \{l\}, \, l \in K}} (r^b_i + x_l) \, . \label{eq:dominion}
\end{equation}
Suppose that $x \le \shapley(x)$. If $K$ is not a dominion, then there exists a state $k \in K$ and an action $\{i,j \} \in A^{(k)}$ such that we have $l \notin K$ for all $\{l \} \in B^{(i)}$. Hence, we have
\[
x_{k} \le (\shapley(x))_k =
\min_{
\substack{a \in A^{(k)}\\
a = \{i, j\}}}
\Bigl(r^a_k + \frac{1}{2}\bigl( \zero
+ \max_{
\substack{
b \in B^{(j)}\\
b = \{l\}}} (r^b_j + x_l)\bigr)\Bigr)  = -\infty \, ,
\]
what gives a contradiction. Therefore, $K$ is a dominion. Let $\tilde{x} \in \R^{\card{K}}$ denote the stratum of $x$ associated with $K$. Observe that~(\ref{eq:dominion}) shows that if $\shapley^{K} \colon \trop^{\card{K}} \to \trop^{\card{K}}$ denotes the Shapley operator of the subgame induced by $K$, then we have $(\shapley(x))_{k} = (\shapley^{K}(\tilde{x}))_{k}$ for every $k \in K$. Let $(\bias, \eigenval) \in \R^{\card{K}} \times \R^{\card{K}}$ denote an invariant half-line of $\shapley^{K}_{|\R^{\card{K}}}$. Suppose that $\eigenval_{k} < 0$ for some $k \in K$. Fix a natural number $N \ge 1$. We have
\[
x_{k} = \tilde{x}_{k} \le (\shapley^{K}(\tilde{x}))_{k} \le  ((\shapley^{K} \circ \shapley^{K})(\tilde{x}))_{k} \le \dots \le ((\shapley^{K} \circ \dots \circ \shapley^{K})(\tilde{x}))_{k} \, ,
\]
where $\shapley^{K}$ is composed $N$ times. Hence $x_{n}/N \le ((\shapley^{K} \circ \dots \circ \shapley^{K})(\tilde{x}))_{k}/N$. By Corollary~\ref{corollary:limit_of_shapley}, the right-hand side of this inequality converges to $\eigenval_{k} < 0$. This implies that $x_{k} = \zero$, what gives a contradiction. Hence $\eigenval_{k} \ge 0$ for every $k \in K$ and $K$ is a winning dominion by Theorem~\ref{theorem:optimal_policies_proof}.

Conversely, suppose that $K$ is a winning dominion. Let $\shapley^{K}$ denote the Shapley operator of the subgame induced by $K$ and let $(\bias, \eigenval)$ be the invariant half-line of $\shapley^{K}_{|\R^{\card{K}}}$. Theorem~\ref{theorem:optimal_policies_proof} shows that $\eigenval_{k} \ge 0$ for all $k \in K$. By the definition, for $\gamma \ge 0$ large enough we have $\shapley^{K}(\bias + \gamma \eigenval) = \bias + (\gamma +1 ) \eigenval \ge \bias + \gamma \eigenval$. Take any such $\gamma$ and denote $\tilde{x} = \bias + \gamma \eigenval$. We extend the point $\tilde{x} \in \R^{\card{K}}$ to $x \in \trop^{n}$ by setting $x_{k} = \tilde{x}_{k}$ for $k \in K$ and $x_{k} = \zero$ otherwise. As previously,~(\ref{eq:dominion}) shows that we have $(\shapley(x))_{k} = (\shapley^{K}(\tilde{x}))_{k}$ for every $k \in K$. In particular, $x_{k} = \tilde{x_{k}} \le (\shapley^{K}(\tilde{x}))_{k} = (\shapley(x))_{k}$ for all $k \in K$. Finally, for every $k \notin K$ we have $x_{k} = \zero \le (\shapley(x))_{k}$.
\end{proof}

The following is an immediate corollary.
\begin{corollary}\label{coro-universal}
Take a Shapley operator $\shapley \colon \trop^{n} \to \trop^{n}$
such that all the initial states $i\in [n]$ of the associated stochastic
mean payoff game have the same mean payoff. Then, the set 
$\{x \in \trop^{n} \colon x \le \shapley(x) \}$ is nontrivial if and only if it contains
a finite vector.
\end{corollary}
One may reinforce the condition of this corollary by requiring
that the mean payoff of the game remains independent
of the initial state for all numerical
values of the payments of the game, the transitions being unchanged.
The latter property admits a convenient combinatorial characterization,
which relies on the {\em recession function} $\hat{\shapley}$ of the Shapley operator:
\[
\hat{\shapley}(x) \coloneqq \lim_{\gamma \to\infty} \gamma^{-1}\shapley(\gamma x) \enspace.
\]
Since $\shapley$ commutes with the addition of a constant vector, it is immediate that $\hat{\shapley}(\alpha,\dots,\alpha)= (\alpha,\dots,\alpha)$ holds
for all $\alpha\in \R$. We call such fixed point of $\hat{\shapley}$ 
{\em uniform}. 
\begin{theorem}[{\cite[Th.~3.1]{ergodicity_conditions}}]\label{th-ergodicity}
Suppose that the recession function $\hat{\shapley}$ has only uniform fixed
points. Then, there exist $u\in \R^n$ and $\lambda\in \R$ such that
$\shapley(u)=\lambda +u$. In particular, the mean payoff of the game
is equal to $\lambda/2$ for all initial states.
\end{theorem}
The condition that $\shapley$ has only uniform fixed points can be checked by finding
invariant sets in directed hypergraphs~\cite{ergodicity_conditions}.
For concrete examples of games arising from spectrahedra, we shall see that this
leads to explicit conditions on the zero/non-zero pattern
of the matrices defining the spectrahedron (Remark~\ref{rk-recession}).

Let us now suppose that $\bo Q^{(0)}, \dots, \bo Q^{(n)} \in \puiseux^{m \times m}$ are Metzler matrices, and let $Q^{(k)} \coloneqq \sval(\bo Q^{(k)})$. By~\cite[Lemma~20]{tropical_spectrahedra} and Theorem~\ref{theorem:generic_metzler}, provided that the matrices $Q^{(k)}$ are generic, the affine spectrahedron 
\begin{align}
 \bspectra\coloneqq \{ \bo x \in \nnpuiseux^n \colon
 \bo Q^{(0)} + \bo x_1  \bo Q^{(1)} + \dots + \bo x_n \bo Q^{(n)}  \loew 0 \}
\label{eq:lastminute}
\end{align}
is nonempty if and only if the tropical Metzler spectrahedron $\spectra(Q^{(0)}, \dots, Q^{(n)})$ contains a point $x$ such that $x_0 \neq -\infty$. Let $\Gamma$ be the stochastic mean payoff game associated arising from the matrices $Q^{(0)}, \dots, Q^{(n)}$. The latter property can be checked using dominions as follows:
\begin{corollary}\label{coro:affine}
The set $\{ x \in \spectra(Q^{(0)}, \dots, Q^{(n)}) \colon x_0 \neq -\infty \}$ is nonempty if and only if there is a winning dominion $\dominion$ in $\Gamma$ such that  $0 \in \dominion$. 
\end{corollary}

\subsection{The archimedean feasibility problem}\label{sec:archimedean}

We next relate the tropical feasibility problem with the archimedean feasibility
problem. For simplicity of exposition, we consider the case of
conic spectrahedra.

We suppose that $\bo Q^{(1)}, \dots, \bo Q^{(n)} \in \puiseux^{m \times m}$ are symmetric Metzler matrices, set
$\bo Q(\bo x)\coloneqq \bo x_1  \bo Q^{(1)} + \dots + \bo x_n \bo Q^{(n)}$,
and consider $\bspectra$ as in~\eqref{eq:lastminute} with $\bo Q^{(0)}=0$.
For any fixed value of the parameter $t \in \R$, we also consider the real spectrahedron $\bspectra(t) \subset \nnR^{n}$ described by $\bo Q^{(1)}(t), \dots, \bo Q^{(n)}(t)$ in a similar manner. We want to study the feasibility problem of $\bspectra(t)$ as $t$ goes to infinity. We denote $\sval(\bo Q^{(k)}) = Q^{(k)}$ for all $k$ and we further suppose that the matrices $Q^{(1)}, \dots, Q^{(n)}$ satisfy Assumption~\ref{assump:well_formed}. The proof is based on the following definition and lemma.

\begin{definition}
For any $\alpha \ge 1$ we define the set $\bspectra_{2,\alpha} \subset \nnpuiseux^{n}$ as
\[
\bspectra_{2, \alpha} \coloneqq \Bigl\{ \bo x \in \nnpuiseux^{n} \colon  \forall i, \bo Q_{ii}(\bo x) \ge 0 \, ,
 \forall i \neq j, \bo Q_{ii}(\bo x)\bo Q_{jj}(\bo x) \ge \alpha(\bo Q_{ij}(\bo x))^{2} \Bigr\} \, .
 \]
\end{definition}

\begin{lemma}\label{lemma:approximate_spectra}
We have the inclusion $\bspectra_{2, (m-1)^{2}} \subset \bspectra \subset \bspectra_{2,1}$.
\end{lemma}

We refer to \cite[Section~5.1]{tropical_spectrahedra} for the proof of Lemma~\ref{lemma:approximate_spectra}. We also introduce a threshold $T > 1$ such that for all $t \geq T$, every series $\bo Q^{(k)}_{ij}(t)$ converges, and the signs of $\bo Q^{(k)}_{ij}(t)$ and $Q^{(k)}_{ij}$ are the same. For any $t \geq T$, we define
\[
\defor(t) \coloneqq \max_{Q^{(k)}_{ij} \neq \zero}\big| \abs{Q^{(k)}_{ij}} - \log_{t}\abs{ \bo Q^{(k)}_{ij}(t)} \big| \, .
\]
By the definition of valuation and order in $\puiseux$ we have $\lim_{t \to \infty} \defor(t) = 0$. Furthermore, for all $k,i,j$ we have $t^{\abs{Q^{(k)}_{ij}} - \defor(t)} \le \abs{\bo Q^{(k)}_{ij}(t)} \le t^{\abs{Q^{(k)}_{ij}} + \defor(t)}$, even if $Q^{(k)}_{ij} = \zero$ (we set $t^{\zero} = 0$). The next theorem relates the feasibility of the spectrahedron $\bspectra(t)$ with the value of the mean payoff game associated with $Q^{(1)}, \dots, Q^{(n)}$ when $t$ is sufficiently large. We point out that this result does not require the genericity assumptions of Theorem~\ref{theorem:generic_metzler}.
\begin{theorem}\label{th:archimedean}
Let $m \ge 2$, and $\gameval$ be the value of the stochastic mean payoff game associated with $Q^{(1)}, \dots, Q^{(n)}$. Let $\lambda \coloneqq \max_k \gameval_k$, and suppose that $\lambda \neq 0$. Take any $t \geq T$ such that $\delta(t) < \abs{\lambda}$ and 
\[
t > (2(m-1)n)^{1/(2\abs{\lambda} - 2\defor(t))} \, .
\]
Then, the spectrahedron $\bspectra(t)$ is nontrivial if and only if $\lambda$ is positive.
\end{theorem}

\begin{proof}
Suppose that $\lambda$ is positive. Theorem~\ref{theorem:spectra_and_game} shows that there is a point $x \in \trop^{n}$, $x \neq \zero$ such that
\[
\forall (i,j) \in [m]^{2}, \  2\lambda + Q_{ij}^{-}(x) \le \frac{1}{2}Q_{ii}^{+}(x) + \frac{1}{2}Q_{jj}^{+}(x) \, .
\]
Thus, we have $t^{4\lambda} t^{2Q_{ij}^{-}(x)} \le t^{Q_{ii}^{+}(x)}t^{Q_{jj}^{+}(x)}$. Take the point $\bo x = (t^{x_{1}}, \dots, t^{x_{n}})$, where $t^{\zero} = 0$. Observe that we have
\[
\forall (i,j) \in [m]^{2}, \ \bo Q_{ij}^{-}(t)(\bo x) \le \sum_{Q_{ij}^{(k)} \in \negtrop} t^{\abs{Q_{ij}^{(k)}} + \defor(t) + x_{k}} \le nt^{Q_{ij}^{-}(x) + \defor(t)}
\]
and
\[
\forall i \in [m], \ \bo Q_{ii}^{+}(t)(\bo x) \ge \sum_{Q_{ii}^{(k)} \in \postrop} t^{Q_{ii}^{(k)} - \defor(t) + x_{k}} \ge t^{Q_{ii}^{+}(x) - \defor(t)}.
\]
Therefore, for all $i \in [m]$ we have
\[
\bo Q_{ii}^{+}(t)(\bo x) \ge \frac{t^{2\lambda-2\defor(t)}}{n} \bo Q_{ii}^{-}(t) \, .
\]
Since $t \ge (2n)^{1/(2\lambda - 2\defor(t))}$, we have $\bo Q_{ii}^{+}(t)(\bo x) \ge 2\bo Q_{ii}^{-}(t)$ and hence $\bo Q_{ii}(\bo x) \ge \frac{1}{2} \bo Q_{ii}^{+}(\bo x)$. Thus, for any $i < j$ we have
\begin{align*}
\bo Q_{ii}(\bo x) \bo Q_{jj}(\bo x) &\ge \frac{1}{4}\bo Q^{+}_{ii}(\bo x) \bo Q^{+}_{jj}(\bo x) \ge \frac{1}{4}t^{Q_{ii}^{+}(x) + Q_{jj}^{+}(x) - 2\defor(t)} \\ 
&\ge \frac{1}{4}t^{2\abs{Q_{ij}(x)} + 4\lambda - 2\defor(t)} \ge \frac{t^{4\lambda - 4\defor(t)}}{4n^{2}}(\bo Q_{ij}(\bo x))^{2} \, .
\end{align*}
Hence, since $t \ge (2(m-1)n)^{1/(2\lambda - 2\defor(t))}$, we get $(\bo Q_{ii}(t)(\bo x))  (\bo Q_{jj}(t)(\bo x)) \ge (m-1)^{2}(\bo Q_{ij}(t)(\bo x))^{2}$ for all $i < j$ and the point $\bo x$ belongs to $\bspectra(t)$ by Lemma~\ref{lemma:approximate_spectra}.

Conversely, suppose that $\lambda$ is negative but $\bspectra(t)$ is nontrivial. Take any nonzero point $\bo x \in \bspectra(t)$ and a point $x \in \trop^{n}$ defined as $x_{k} = \log_{t}( \bo x_{k})$ for all $k \in [n]$ (where $\log_{t}(0) = \zero$). Since $\lambda$ is negative, Theorem~\ref{theorem:spectra_and_game} shows that for every fixed $\varepsilon > 0$ there is a pair $(i,j) \in [m]^{2}$ such that
\[
2\lambda + \varepsilon + Q_{ij}^{-}(x) > \frac{1}{2}Q_{ii}^{+}(x) + \frac{1}{2}Q_{jj}^{+}(x) \, .
\]
Hence $t^{4\lambda + 2\varepsilon}t^{2Q_{ij}^{-}(x)} > t^{Q_{ii}^{+}(x)}t^{Q_{jj}^{+}(x)}$. Similarly to the previous case, observe that we have
\[
\bo Q_{ij}^{-}(t)(\bo x) \ge \sum_{Q_{ij}^{(k)} \in \negtrop} t^{Q_{ij}^{(k) - \defor(t)} + x_{k}} \ge t^{Q_{ij}^{-}(x) - \defor(t)}
\]
and
\[
\bo Q_{ii}^{+}(t)(\bo x) \le \sum_{Q_{ii}^{(k)} \in \postrop} t^{Q_{ii}^{(k)} + \defor(t) + x_{k}} \le nt^{Q_{ii}^{+}(x) + \defor(t)} \, .
\]
Therefore, we have
\[
t^{4\lambda + 2\varepsilon + 2\defor(t)} (\bo Q^{-}_{ij}(t)(\bo x))^{2} > t^{Q_{ii}^{+}(x)}t^{Q_{jj}^{+}(x)} \ge \frac{t^{-2\defor(t)}}{n^2}(\bo Q^{+}_{ii}(t)(\bo x))  (\bo Q^{+}_{jj}(t)(\bo x)) \, .
\]
Thus, if we take $\varepsilon$ such that $2\lambda + \varepsilon + 2\defor(t) < 0$ and $t \ge n^{1/\abs{2\lambda + \varepsilon + 2\defor(t)}}$, we have $(\bo Q^{-}_{ij}(t)(\bo x))^{2} > (\bo Q^{+}_{ii}(t)(\bo x))  (\bo Q^{+}_{jj}(t)(\bo x))$, what gives a contradiction.
\end{proof}
\begin{remark}
It is easy to see from the proof that if the matrices $\bo Q^{(k)}$ are diagonal (which holds, in particular, if $m = 1$), then the term $2(m-1)$ is not needed, and the bound for $t$ takes the form $t > n^{1/(2\abs{\lambda} - 2\defor(t))}$.
\end{remark}

\begin{example}
We can apply this bound to the example presented in Section~\ref{section:illustration}. Let us take the simplest lift of matrices $Q^{(1)}, Q^{(2)}, Q^{(3)}$, namely $\bo Q^{(k)}_{ij} = \sign(Q^{(k)}_{ij})t^{Q^{(k)}_{ij}}$. In this case we have $\defor(t) = 0$ for all $t > 1$. Moreover, the computations presented in Example~\ref{ex:payoff_computation} show that $\lambda = 1/56$. Thus, if we take $t > 12^{28}$, then the spectrahedron $\bspectra(t)$ is nontrivial. In this special instance, the lower bound provided by Theorem~\ref{th:archimedean} is large because we chose a game that is nearly singular: the mean payoff $\lambda$ is close to $0$. There are, however, other instances of nonarchimedean spectrahedra for which $\lambda$ is of order one, leading to a $(mn)^{O(1)}$ bound for $t$. 
\end{example}

\section{Algorithms}
\label{sec-algo}

In this section, we discuss algorithms that can solve \Tropicalsdfp{} thanks to the equivalence with \Smpg{} established in Theorem~\ref{theorem:simple_games}.

\subsection{Complexity bounds}
We first derive complexity bounds for \Tropicalsdfp{}.
Let $L$ denote the maximal number of bits needed to encode an entry from matrices $Q^{(1)}, \dots, Q^{(n)}$. An exponential bound for \Tropicalsdfp{} is achieved by a naive algorithm that enumerates all policies of one player and uses linear programming to solve the remaining $1$-player game. 

\begin{theorem}\label{th:strategy_enum}
There is an algorithm that solves \textsc{Tmsdfp} in 
\[
\min\{n^{m}, m^{2n}\}\Poly(m, n, L)
\]
arithmetic operations.
\end{theorem}
\begin{proof}
Take the matrices $Q^{(1)}, \dots, Q^{(n)}$ and consider the associated stochastic game $\Gamma$. If $m^{2n} \le n^{m}$, then we do the following: for each policy $\sigma$ of Player Min, we consider a $1$-player game $\Gamma^{\sigma}$ induced by fixing $\sigma$. The value of $\Gamma^{\sigma}$, denoted $\gameval^{\sigma}$, can then be found in polynomial time by linear programming (see \cite[Section 2.9]{filar_vrieze}). Moreover, for every fixed state $k \in [n]$, the value $\gameval_{k}$ of $\Gamma$ can be found by taking the minimum over $\sigma$ of values $\gameval^{\sigma}_{k}$. Now, by Theorem~\ref{theorem:spectra_and_game}, the tropical Metzler spectrahedron $\spectra$ associated with $Q^{(1)}, \dots, Q^{(n)}$ is nontrivial if and only if $\gameval$ has at least one nonnegative entry. Since Player Min has at most $m^{2n}$ policies in $\Gamma$, this procedure takes at most $m^{2n}\Poly(m, n, L)$ arithmetic operations. If $m^{2n} > n^{m}$, then we do the analogous operation for Player Max.
\end{proof}

The best currently known bounds for \Tropicalsdfp{} can be derived from the
interpretation of stochastic games as LP-type problems.
These bounds are randomized subexponential,
~\cite{halman}, and also~\cite{hansen_zwick_random_facet} for a recent improvement. A distinctive feature of this approach is that it works in \textit{strongly} subexponential time, i.e., its arithmetic complexity does not depend on $L$.

\begin{theorem}\label{th:randomized}
There is a randomized algorithm that solves \textsc{Tmsdfp} in expected 
\[
\exp(O(\sqrt{(m + n)\log(m + n)}))
\]
arithmetic operations.
\end{theorem}
\begin{proof}
As in the proof of Theorem~\ref{th:strategy_enum}, take the matrices $Q^{(1)}, \dots, Q^{(n)}$ and consider the associated game $\Gamma$. Now, we apply to $\Gamma$ the first step of reduction presented in \cite[Lemma~1 and Lemma~3]{andersson_miltersen} (i.e., the step presented in Fig.~2 of the latter reference). This converts $\Gamma$ in a strongly polynomial complexity to a game $\Gamma_{1}$, which has the same optimal policies as $\Gamma$ and has the form considered by
Halman~\cite{halman}. Moreover, $\Gamma_{1}$ has at most $n$ min vertices and at most $m$ max vertices (in Halman's terminology). Furthermore, Halman~\cite[Theorem~4.1]{halman} gives an algorithm that finds optimal policies for games belonging to his class in the expected $\exp(O(\sqrt{p\log p}))$ arithmetic operations, where $p$ is the total number of min and max vertices. Hence, a pair of optimal policies in $\Gamma_{1}$ (which is also optimal in $\Gamma$) can be found in the expected 
\[
\exp(O(\sqrt{(m + n)\log(m + n)}))
\]
arithmetic operations. Once a pair of optimal policies is known, the value of $\Gamma$ can be found in strongly polynomial complexity by Remark~\ref{remark:payoff_strongly_polynomial}. As in the proof of Theorem~\ref{th:strategy_enum}, we use the value of $\Gamma$ to decide whether the tropical Metzler spectrahedron $\spectra$ is trivial.
\end{proof}

\subsection{Value iteration}
We now present an algorithm with poorer theoretical bounds,
but which is well adapted practically to some large scale instances: value iteration. This algorithm may be thought of as a nonlinear
analogue of the power algorithm to compute the dominant eigenvalue of a matrix.
Its advantage lies in scalability. 
This leads to a procedure called \textsc{CheckFeasibility}, provided in Figure~\ref{fig:valueiter} and which checks the existence of
a sub-harmonic vector of a Shapley operator of the form~\eqref{e-def-shapley}.
By Lemma~\ref{lemma:spectrahedron_encodes_shapley}, this is equivalent to finding a point in a tropical spectrahedron. 

\begin{figure}
\begin{small}
\begin{algorithmic}[1]
\Procedure {CheckFeasibility}{$\shapley$, 
$\varepsilon$} \\ \Lcomment{$\shapley$ a Shapley operator from $\trop^n$ to $\trop^n$, 
 $\varepsilon>0$ a numerical precision.}
\State $u \coloneqq 0 \in \R^n$, $\bias \coloneqq 0\in \R^n$
\While{$\max_k u_k>-\varepsilon$ and $\min_k u_k<\varepsilon$}
\State $\bias \coloneqq \max(\bias,u)$, $u \coloneqq \shapley(u)$  

\Lcomment{
The operation max on vectors is understood entrywise.}
\EndWhile
\If{$\max_k u_k\leq -\varepsilon$}
\State There is no vector $x\in\trop^n$,  $x\not\equiv-\infty$, such that $x\leq \shapley(x)$.
\Else
\State The vector $\bias$ satisfies $\bias\leq \shapley(\bias)$.
\EndIf
\EndProcedure
\end{algorithmic}
\end{small}
\caption{Checking the feasibility of a tropical 
semidefinite problem 
by value iteration.} \label{fig:valueiter}
\end{figure}

We next establish the correctness of the \textsc{CheckFeasibility} algorithm,
under the assumption of Corollary~\ref{coro-universal}, that
all the initial states in the mean payoff game associated
with the Shapley operator $\shapley$ have the same value $\lambda$.
It follows from Section~\ref{sec:game_equivalence} (Corollary~\ref{corollary:same_payoff}) that every  stochastic 
mean payoff game reduces polynomially to a game with this property.

\begin{theorem}\label{prop-algo}
Suppose that all the initial states of the game with Shapley operator $\shapley$
have the same value. Suppose in addition that this value,
$\lambda$, differs from $0$.
Then, for all choices of $\varepsilon>0$,
the procedure {\sc CheckFeasibility} terminates and
is correct.
\end{theorem}

\begin{proof}

It will be convenient to use the following notation, for a vector
$z\in \R^n$:
\[
\mathbf{t}(z) \coloneqq \max_{i} z_i \, , \qquad \mathbf{b}(z) \coloneqq \min_{i} z_i \, ,
\]
so that the halting condition reads
\[
\mathbf{t}(u)\leq -\varepsilon \, \text{ or } \, \mathbf{b}(u)\geq \varepsilon \, .
\]
The sequences $u^{(0)},u^{(1)},\dots$
and $\bias^{(0)},\bias^{(1)},\dots$ generated by the algorithm implemented in exact arithmetic satisfy, for $\ell\geq 1$,
\[
u^{(\ell)}=\shapley^\ell(0), \qquad  \bias^{(\ell)} = \max(0,\shapley(0),\dots,\shapley^{\ell-1}(0)) \, .
\]
We know from Corollary~\ref{corollary:limit_of_shapley}
and Theorem~\ref{theorem:optimal_policies_proof} 
that the limit $\lim_\ell u^{(\ell)}/2\ell$ coincides with the mean payoff vector $\gameval$. The assumptions of the present theorem imply that $\gameval$ is a constant vector with nonzero entries, i.e.,
$\gameval=(\lambda,\dots,\lambda)^\top$ where $\lambda\neq 0$. Therefore, 
$\lim_\ell \mathbf{b} (u^{(\ell)})/\ell = \lim_\ell \mathbf{t} (u^{(\ell)})/\ell=\lambda$. In particular,
for $\ell$ sufficiently large, we have either $\mathbf{b} (u^{(\ell)}) >\varepsilon$,
or $\mathbf{t}( u^{(\ell)}) <-\varepsilon$, depending on the sign of $\lambda$,
and so the halting condition is ultimately satisfied. If $\mathbf{t} (u^{(\ell)}) <-\varepsilon$,
then $\shapley^\ell(0) \leq -\varepsilon e$, where $e = (1,1,\dots,1)$.
Since $\shapley$ is order preserving and commutes with the addition
of a constant vector, we get, after an immediate induction,
that $\shapley^{\ell p}(0)\leq -p\varepsilon e$ holds for all $p \ge 1$, and
therefore, the mean payoff vector of the game, which coincides with
$\lim_{p} \shapley^{\ell p}/(2\ell p)$, has negative entries.
By the Collatz--Wielandt property~\eqref{pty1:collatz-wielandt},
this implies that the tropical spectrahedron
$\{x:\, x\leq \shapley(x)\}$ is reduced to the trivial vector with
$-\infty$ entries. If $\mathbf{b} (u^{(\ell)}) > \varepsilon$, then the dual argument shows that
the mean payoff of the game $\lim_{p}\shapley^{\ell p}/(2\ell p)$
has positive entries. Moreover, in this case we get
\begin{align*}
\shapley(\bias)&= \shapley(\max(0,\shapley(0),\dots,\shapley^{\ell-2}(0),\shapley^{\ell-1}(0)))\\
&\geq 
\max(\shapley(0),\shapley^2(0),\dots,\shapley^{\ell-1}(0), \shapley^{\ell}(0)) \\
&\geq 
\max(\shapley(0),\shapley^2(0),\dots,\shapley^{\ell-1}(0), \varepsilon e)\\
&\geq \max(\shapley(0),\shapley^2(0),\dots, \shapley^{\ell-1}(0), 0) = \bias \, ,
\end{align*}
which implies that $v$ belongs to the tropical spectrahedron
$\{x:\, x\leq \shapley(x)\}$. 

We note that the algorithm still terminates, and provides %
a correct yes/no answer, if the operator
$\shapley$ is evaluated in fixed precision arithmetic,
in such a way that $u$ never differs from its
true value by more than $\varepsilon$ in the sup norm $\|\cdot\|$.
Indeed, let $\tilde u^{(0)},\tilde u^{(1)},\dots$
denote the successive approximate values of the variables $u$
which are computed.
Our assumption entails that
$\|\tilde u^{(\ell)} - u^{(\ell)}\| \leq \varepsilon$ at every step
$\ell$ of the algorithm. The termination proof
relies on the fact that either $\mathbf{b}(u^{(\ell)})$ tends
to infinity, or $\mathbf{t}(u^{(\ell)})$ tends to $-\infty$.
Since $\|\tilde u^{(\ell)} - u^{(\ell)}\| \leq \varepsilon$,
the analogous property is still satisfied by 
$\tilde u^{(\ell)}$, and so, the algorithm 
implemented with finite precision does terminate.

Now, if $\ell^*$ is the step at which
the procedure terminates, we have 
either $\mathbf{b}(\tilde u^{({\ell^*})})>\varepsilon$
or $\mathbf{t}(\tilde u^{({\ell^*})})<-\varepsilon$,
which entails that either 
$\mathbf{b}(u^{({\ell^*})})>0$
or $\mathbf{t}(u^{({\ell^*})})<0$. Reasoning
as above, we deduce that the mean payoff
of the game is positive in the first situation
and negative in the second one. Hence,
the algorithm still decides correctly the
triviality of the spectrahedron. 
\end{proof}

\begin{remark}
The situation in which the mean payoff is zero is degenerate: then, an infinitesimal perturbation of the entries of the matrices can make the spectrahedron trivial or nontrivial.
Value iteration cannot naturally handle such degenerate
situations, which can be solved by different methods, like the policy
iteration algorithm presented in~\cite{cdcdetournay}, which is based on the solution
of a finite sequence of linear systems, and can be implemented
in exact arithmetic. 
\end{remark}
\begin{remark}
Every iteration (while loop) of the procedure {\sc CheckFeasibility} takes a time $O(nm^2)$, which is {\em linear} in the size of the input.
The number of iterations can only be bounded
by an exponential in the size of the input. However,
our benchmarks indicate that the algorithm is fast when the instance is ``far from being degenerate''.
\end{remark}
\begin{remark}\label{rk-recession}
The condition that all the initial states have the same mean payoff
can be checked by appealing to Theorem~\ref{th-ergodicity}
involving the recession function $\hat{\shapley}$ of the Shapley
operator. For instance, suppose that $\shapley$ arises
from a Metzler tropical spectrahedron given by matrices whose
entries $Q_{ij}^{(k)}$ are all finite.
Then, one  can verify that the $k$-th coordinate map of $\hat{\shapley}$
is given by
\[
\hat{\shapley}_k(x)= \min_{i<j} \frac{1}{2} \bigl(\max_{Q^{(l)}_{i i} \in \postrop} x_l
 + \max_{Q^{(l)}_{j j} \in \postrop}  x_l\bigr) \, ,	
\]
and is therefore independent of the choice of $k\in [n]$. 
Then, it follows from Theorem~\ref{th-ergodicity} that $\hat{\shapley}$ as only fixed points of the form $(\alpha,\dots,\alpha)$, which implies that the first condition of Theorem~\ref{prop-algo} is satisfied.
\end{remark}
\begin{remark}
By Theorem~\ref{theorem:generic_metzler}, the procedure 
{\sc CheckFeasibility} allows us to verify the feasibility
of a Metzler nonarchimedean spectrahedron if the entries
of the matrices $Q_{ij}^{(k)}$ are finite and generic.
\end{remark}

\begin{example}
Procedure {\sc CheckFeasibility} applied to the nonarchimedean spectrahedron of Section~\ref{section:illustration}, with $\varepsilon=10^{-8}$,
terminates in 20 iterations. It returns a vector having floating point entries \[(1.05612\dots,0.0204082\dots,1.12755\dots) \, .\] 
When converted into a vector with rational entries, it reads
\[
v \coloneqq \Bigl(1107425 \times 2^{-20}, 42799 \times 2^{-21}, 4729289 \times 2^{-22} \Bigr) \, .
\]
We have checked using exact precision arithmetic over rationals (provided by the GNU multiple precision arithmetic library, \url{https://gmplib.org/}) that this vector lies in the interior of the tropical spectrahedron shown in Figure~\ref{fig:tropical_spectrahedron}, \ie, $v_i < \shapley(v_i)$ for $i = 1,2,3$.  Following Lemma~\ref{lemma:lift}, the vector $\bo x \coloneqq (t^{v_1},t^{v_2},t^{v_3})$ fulfills $\bo Q(\bo x)\succeq 0$.
\end{example}

\begin{figure}
\begin{center}
\includegraphics{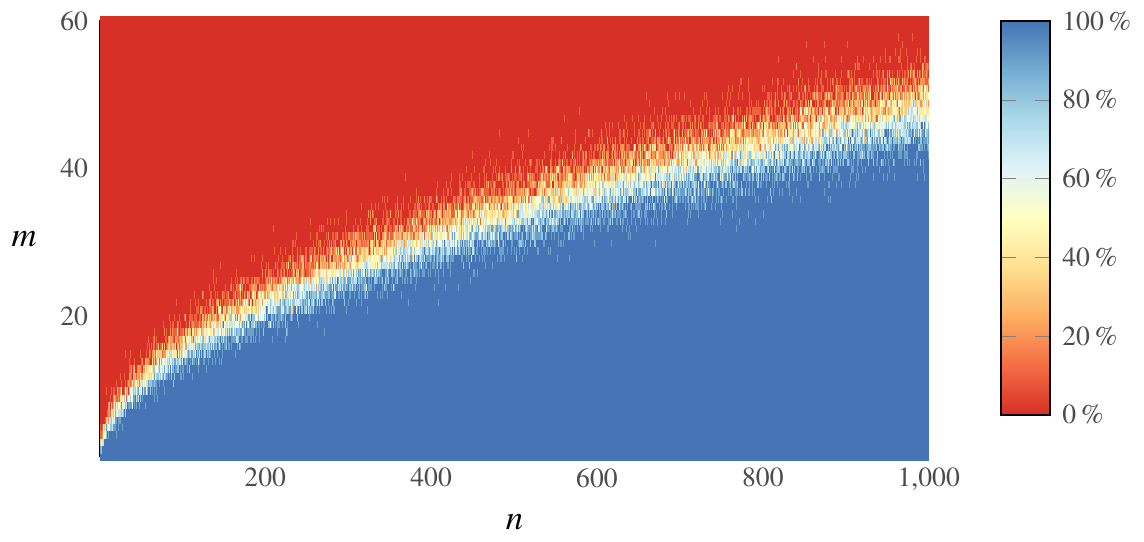}
\end{center}
\caption{Phase transition between feasibility and infeasibility of spectrahedra. For each $(n,m)$, the color scheme reports the ratio of feasible instances among $10$ samples.}\label{fig:phase_transition}
\end{figure}

We report in Table~\ref{table:bench} experimental results for
different values of $(n,m)$.
We chose all the $|Q^{(k)}_{ij}|$, for $i\leq j$, to be independent 
random variables uniformly distributed on $[0,1]$.
Moreover, the diagonal coefficients $Q^{(k)}_{ii}$
were chosen to have a positive tropical sign (they belong to $\postrop$). 
We took $\varepsilon=10^{-8}$ (the performance was similar for $\varepsilon=10^{-6}$ or $\varepsilon=10^{-10}$). Our experiments were obtained using a C program,
distributed 
as an ancillary file attached to this arXiv manuscript 
for reproducibility purposes.\footnote{The ancillary file can be downloaded from \url{http://arxiv.org/src/1603.06916/anc}.}
This program was compiled under Linux with gcc -O3, and executed on a single core of an Intel(R) i7-4600U CPU at 2.10~GHz with 16~GB RAM. We report the average execution
time over 10 samples for every value of $(n,m)$.
The number of iterations did not exceed 731 on this benchmark, and,
for most $(n,m)$, it was limited to a few units.
Indeed, random instances
exhibit experimentally a phase transition, as shown in Figure~\ref{fig:phase_transition}: for a given $(n,m)$,
the system is either feasible
with overwhelming probability, 
or infeasible
with overwhelming probability, 
unless $(n,m)$ lies in a tiny region of the parameter space. 
Value iteration quickly decides feasibility,
except in regions close to the phase transition. This 
explains why the execution time does not increase
monotonically with $(n,m)$ in our experiments (we included
both easy and hard values of $(n,m)$). 

\begin{table}
\small
\begin{center}
\begin{tabular}{@{}c@{\hskip 10pt}c@{\hskip 10pt}c@{\hskip 10pt}c@{\hskip 10pt}c@{\hskip 10pt}c@{}}
$(n,m)$&$(50,10)$&$(  50, 40)$
&$( 50, 50)$&$( 50, 100)$&$( 50, 1000)$\\
time& 
0.000065   &
0.000049  &
0.000077  &
0.000279  &
0.026802
\\\hline 
$(n,m)$&$( 100,10)$& $(100,15)$ & $( 100,80)$
&$( 100,100)$&$( 100,1000)$\\
time&   0.000025    & 0.000270  & 0.000366    
    &0.000656   &   0.053944
\\\hline 
$(n,m)$&$( 1000, 10)$&$( 1000,50)$&$( 1000,100)$&$( 1000, 200)$&$( 1000, 500)$\\
time   &  0.000233   & 0.073544     & 0.015305& 0.027762 & 0.148714\\\hline 
$(n,m)$&$( 2000,10)$&$(2000,70)$&$( 2000, 100)$ & $(10000,150)$& $(10000,400)$\\
time  & 0.000487   & 1.852221  &0.087536 &19.919844 & 2.309174
\end{tabular}
\end{center}
\vspace*{0.1cm}
\caption{Execution time (in sec.) of Procedure {\sc CheckFeasibility}
on random instances.} \label{table:bench}
\end{table}

\section{Equivalent forms of stochastic mean payoff games problem}\label{sec:game_equivalence}

In this section we present the (algorithmically) equivalent forms of stochastic mean payoff games, as mentioned in Remark~\ref{remark:general_games}. In order to do that, we need to introduce the notions of \emph{simple stochastic games} and \emph{stopping games}.

We start by defining the class of stopping games. We say that a pair of states $(i, k) \in [m] \times [n]$ is a \emph{sink} if, when $i$ or $k$ is reached, the game loops forever between these two states. More formally, we have $A^{(k)} = \{ \{i\}\}, B^{(i)} = \{\{k \} \}$. We say that the game is \emph{stopping} if it has at least one sink and the probability that the game will reach a sink is equal one for every choice of policies $(\sigma, \tau)$ and every initial state. Note that if a game is stopping, then only the payoffs in sinks are important to determine its solution. Indeed, if we denote the sinks by $(i_{1}, k_{1}), \dots, (i_{p}, k_{p})$ and the game is stopping, then the payoff of Player Max is given by
\[
g_{k}(\sigma, \tau) = \sum_{s = 1}^{p}(r^{\{i_s\}}_{k_s} + r^{\{k_s\}}_{i_i})\mu_{s}(k, \sigma, \tau),
\]
where $\mu_{s}(k, \sigma, \tau)$ is the probability that the game starting from $k$ reaches the sink $(i_{s}, k_{s})$ if the players use the policies $(\sigma, \tau)$. This expression depends only on payoffs in sinks.

Now, we introduce the class of simple stochastic games. We say that the stochastic game is \emph{simple} if its set of states can be divided into three classes: states controlled by Player Min, Player Max, and Nature. Players Min and Max have only deterministic choices and Nature chooses the next state by tossing a coin. (To be coherent with the previous definition of stochastic game, we assume that Player Min controls the states of Nature --- but she has no other choice than to toss a coin.) Formally, we suppose that for every $k \in [n]$ and every $a \in A^{(k)}$ we have the implication $\abs{a} = 2 \implies \abs{A^{(k)}} = 1$. Moreover, a simple game has two sinks: one with payoff $1$ and the other with payoff $0$. All other payoffs are equal to $0$. 

It may seem that solving simple games is indeed simpler that solving games in their full generality. 
Andersson and Miltersen~\cite{andersson_miltersen} have shown that this is not the case. Let \Smpgcomp{} denote the problem of finding the value and a pair of optimal policies in a stochastic mean payoff game.
\begin{theorem}[\cite{andersson_miltersen}]\label{theorem:smpg_to_ssg_val}
\Smpgcomp{} is poly-time equivalent to the problem of finding the values of stopping simple stochastic games.
\end{theorem}
Note that this theorem was originally 
without the word ``stopping'', but this is what the authors actually showed in the latter reference. 
As already mentioned in Remark~\ref{remark:general_games}, this result is valid for a much wider class of games than those considered in this work. 
Andersson and Miltersen
defined stopping simple stochastic games in a slightly more general way --- in their version, the players can make multiple moves in a row. Nevertheless, observe that we can always add dummy states to a stopping simple stochastic game (i.e., states in which player has only one action) and obtain an equivalent game that belongs to the class considered here. This shows that, from the algorithmic point of view, these classes are equivalent.

Moreover, observe that if the game is both simple and stopping, then we can change its payoffs in sinks --- instead of payoffs equal to $0$ and $1$, we can demand them to be equal to $-1$ and $1$. This does not change the optimal policies of the game and acts as an affine transformation on the value vector. Henceforth, we assume that payoffs in sinks of stopping simple stochastic games are equal to $-1$ and $1$. We now show that the computational problem of finding values of simple games can be reduced to the decision problem. By \Smpgi{} we will denote the problem of deciding if a given state $k \in [n]$ is winning in the stochastic mean payoff game.
\begin{lemma}\label{lemma:computation_to_decision}
\Smpgi{} restricted to stopping simple stochastic games is poly-time equivalent to \Smpgcomp{}.
\end{lemma}
\begin{proof}
It is obvious that \Smpgi{} can be reduced to \Smpgcomp. We will show the opposite reduction. By Theorem~\ref{theorem:smpg_to_ssg_val},  \Smpgcomp{} is poly-time reducible to the problem of finding values of stopping simple stochastic games. Fix such a game and let $\gameval \in [-1,1]^{n}$ denote its value.

First, we show an auxiliary reduction. Fix a rational number $\alpha \in [-1, 1] \cap \Q$ and an initial state $k \in [n]$. Suppose that we want to decide if $\gameval_{k} \ge \alpha$. We will show that this is poly-time reducible (poly-time in the size of the game and the number of bits needed to encode $\alpha$) to \Smpgi. If $\alpha = -1$, then the answer is ``yes''. If $\alpha > -1$, then we may modify the game as follows: we suppose that when the sink with payoff $1$ is reached, the game does not start to loop, but instead moves with probability $1 - \frac{1}{\alpha+1}$ to the sink with payoff $-1$ and with probability $\frac{1}{\alpha+1}$ to the (newly created) sink with payoff $1$. Denote the value of the modified game by $\tilde{\gameval}$. We have $\tilde{\gameval}_{k} \ge 0 \iff \gameval_{k} \ge \alpha$. The modified game is stopping but not simple. Nevertheless, we may apply the construction of Zwick and Paterson~\cite[remarks preceding Theorem 6.1]{zwick_paterson} and obtain (in poly-time) a new game, which is stopping, simple, and has $\tilde{\gameval}_{k}$ as the value at state $k$. This gives the auxiliary reduction. 

Second, \cite[Lemma~1]{auger_strozecki} (which improves \cite[Lemma~2]{condon}) shows that $\gameval_{k}$ is a rational number of form $a/b \in [-1, 1]$, where $a,b$ are integers and $0 \le b \le 3^n$. Thus, $\gameval_{k}$ can be found by the Kwek--Mehlhorn algorithm~\cite{kwek_mehlhorn}, using polynomially many queries to the oracle given by our auxiliary reduction.
\end{proof}

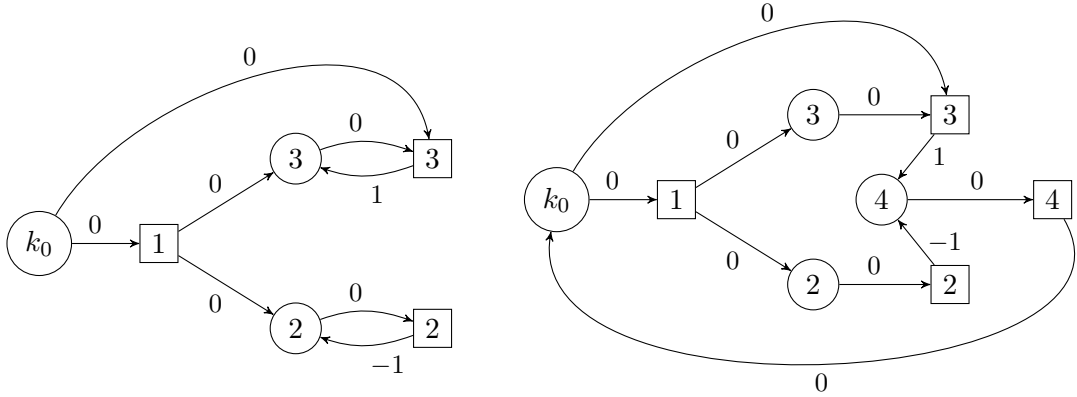
\begin{figure*}[t]
\begin{center}
\begin{minipage}{0.45\textwidth}
\centering
\begin{tikzpicture}[scale=0.9,>=stealth',row/.style={draw,circle,minimum size=0.5cm},col/.style={draw,rectangle,minimum size=0.5cm},av/.style={draw, circle,fill, inner sep = 0pt,minimum size = 0.2cm}]
\node[row] (i1) at (2.25,1.25) {$k_{0}$};
\node[row] (i2) at (6,0) {$2$};
\node[row] (i3) at (6,2.5) {$3$};

\node[col] (j2) at (4,1.25) {$1$};
\node[col] (j3) at (8,2.5) {$3$};
\node[col] (j4) at (8,0) {$2$};

\draw[->] (i1) to node[above left=0ex and -0.5ex, font=\small] {$0$} (j2);
\draw[->] (i3) to[out=20, in=160] node[above left=0ex and -0.5ex, font=\small] {$0$} (j3);
\draw[->] (j3) to[out=-160, in=-20] node[below right=0ex and -0.5ex, font=\small] {$1$} (i3);
\draw[->] (i2) to[out=20, in=160] node[above left=0ex and -0.5ex, font=\small] {$0$} (j4);
\draw[->] (j4) to[out=-160, in=-20] node[below right=0ex and -0.5ex, font=\small] {$-1$} (i2);

\draw[->] (i1) to[out = 60, in =100] node[above left=0ex and -0.5ex, font=\small] {$0$} (j3);

\draw[->] (j2) to node[below left=0ex and -0.5ex, font=\small] {$0$} (i2);
\draw[->] (j2) to node[above left=0ex and -0.5ex, font=\small] {$0$} (i3);

\end{tikzpicture}
\end{minipage}\hfill
\begin{minipage}{0.54\textwidth}
\begin{center}
\begin{tikzpicture}[scale=0.9,>=stealth',row/.style={draw,circle,minimum size=0.5cm},col/.style={draw,rectangle,minimum size=0.5cm},av/.style={draw, circle,fill, inner sep = 0pt,minimum size = 0.2cm}]
\node[row] (i1) at (2.25,1.25) {$k_{0}$};
\node[row] (i2) at (6,0) {$2$};
\node[row] (i3) at (6,2.5) {$3$};

\node[row] (i4) at (7,1.25) {$4$};

\node[col] (j2) at (4,1.25) {$1$};
\node[col] (j3) at (8,2.5) {$3$};
\node[col] (j4) at (8,0) {$2$};

\node[col] (j5) at (9.5,1.25) {$4$};

\draw[->] (i1) to node[above left=0ex and -0.5ex, font=\small] {$0$} (j2);
\draw[->] (i3) to node[above left=0ex and -0.5ex, font=\small] {$0$} (j3);
\draw[->] (j3) to node[right=0ex and 0.5ex, font=\small] {$1$} (i4);
\draw[->] (i2) to node[above left=0ex and -0.5ex, font=\small] {$0$} (j4);
\draw[->] (j4) to node[right=0ex and 0.1ex, font=\small] {$-1$} (i4);

\draw[->] (i1) to[out = 60, in =100] node[above left=0ex and -0.5ex, font=\small] {$0$} (j3);

\draw[->] (j2) to node[below left=0ex and -0.5ex, font=\small] {$0$} (i2);
\draw[->] (j2) to node[above left=0ex and -0.5ex, font=\small] {$0$} (i3);

\draw[->] (i4) to node[above right=0ex and -0.8ex, font=\small] {$0$} (j5);

\draw[->] (j5) to[out = -60, in =-100] node[below left=0ex and -0.5ex, font=\small] {$0$} (i1);

\end{tikzpicture}
\end{center}
\end{minipage}
\end{center}
\caption{Transformation of stopping simple stochastic games. Circle states are controlled by Player Min. Numbers indicate payoffs received by Player Max after each move.}\label{fig:transforming_stopping_game}
\end{figure*}

Finally, we want to show that \Smpgi{} is poly-time reducible to \Smpg. This requires an auxiliary construction which is presented in Figure~\ref{fig:transforming_stopping_game}. We take a stopping simple stochastic game $\Gamma$, fix an initial state $k_{0} \in [n]$ and suppose that the sinks of $\Gamma$ are indexed as $(n - 1, m - 1)$ (sink with payoff $-1$) and $(n, m)$ (sink with payoff $1$). Now, we modify the game as follows: we add two states, $n + 1$ (controlled by Player Min) and $m + 1$ (controlled by Player Max). At $n + 1$, Player Min has only one possible action: to go to $m + 1$; after this action Player Min pays $0$ to Player Max. Moreover, at $m + 1$ Player Max also has only one action: to go to $k_{0}$; after this action Player Max receives $0$ from Player Min. Finally, we modify the sinks of $\Gamma$ as follows: at $m - 1$ (resp.\ $m$) Player Max has only one possible action: to go to $n + 1$; after this action he receives $-1$ (resp.\ $1$) from Player Min. Denote the modified game by $\overbar{\Gamma}$. By construction, it is quite intuitive that the value of $\overbar{\Gamma}$ does not depend on the initial state and that the state $k_{0}$ is winning in $\overbar{\Gamma}$ if and only if it is winning in $\Gamma$. To prove this formally, we use Theorem~\ref{theorem:characterization_of_payoff}. Henceforth, by $\Gamma^{\sigma, \tau}$ we denote the $0$-player game obtained from $\Gamma$ by fixing a pair of policies $(\sigma, \tau)$.
\begin{corollary}\label{corollary:same_payoff}
The state $k_{0}$ is winning in $\Gamma$ if and only if it is winning in $\overbar{\Gamma}$. Moreover, the value of $\overbar{\Gamma}$ is independent of the initial state.
\end{corollary}
\begin{proof}
First, observe that there exists a natural bijection between policies of $\Gamma$ and $\overbar{\Gamma}$. Hence, we will use the same letters to denote policies in both games. Fix a pair of policies $(\sigma, \tau)$. Let $g(\sigma, \tau)$ (resp.\ $\overbar{g}(\sigma, \tau)$) denote the payoff of Player Max in $\Gamma^{\sigma, \tau}$ (resp.\ in $\overbar{\Gamma}^{\sigma, \tau}$). Since $\Gamma$ was stopping, the $0$-player game $\overbar{\Gamma}^{\sigma, \tau}$ obtained from $\overbar{\Gamma}$ by fixing the policies $(\sigma, \tau)$ has only one recurrent class and $k_{0}$ belongs to this class. By Theorem~\ref{theorem:characterization_of_payoff} we see that $\overbar{g}(\sigma, \tau)$ is constant for all initial states. Since $(\sigma, \tau)$ are arbitrary, this shows that the value of $\overbar{\Gamma}$ does not depend on the choice of initial state. Furthermore, Theorem~\ref{theorem:characterization_of_payoff} shows that $\overbar{g}_{k_{0}}(\sigma, \tau) = g_{k_{0}}(\sigma, \tau)/\theta_{k_{0}}$, where $\theta_{k_{0}} = \theta_{k_{0}}(\sigma, \tau)$ is the expected time of first return to $k_{0}$.

Now, suppose that the value of $\Gamma$ starting from $k_{0}$ is higher or equal than $0$. Let $\tau^{*}$ denote the optimal policy of Player Max in $\Gamma$. For any policy $\sigma$ of Player Min we have $g_{k_{0}}(\sigma, \tau^{*}) \ge 0$ and hence $\overbar{g}_{k_{0}}(\sigma, \tau^{*}) \ge 0$. Therefore, $\tau^{*}$ is a winning (but not necessarily optimal) policy for Player Max in $\overbar{\Gamma}$ starting from $k_{0}$. Thus, the value of $\overbar{\Gamma}$ starting from $k_{0}$ is higher or equal than $0$. The opposite implication is analogous.
\end{proof}

\begin{corollary}\label{cor:comp_to_decision}
\Smpgcomp{} is poly-time equivalent to \Smpg{} restricted to games with payoffs in $\{-1, 0, 1\}$. 
\end{corollary}
\begin{proof}
By Lemma~\ref{lemma:computation_to_decision}, \Smpgcomp{} can be reduced to \Smpgi{} restricted to stopping simple stochastic games. By the construction described above and Corollary~\ref{corollary:same_payoff}, \Smpgi{} for stopping simple stochastic games can be reduced to \Smpg{} restricted to games with payoffs in $\{-1, 0, 1\}$. The opposite reduction is trivial.
\end{proof}

\begin{corollary}\label{cor:equivalence_pseudopolynomial}
\Smpg{} restricted to games with payoffs in $\{-1, 0, 1\}$ is poly-time equivalent to \Smpg{} for general games.
\end{corollary}
\begin{proof}
\Smpg{} is trivially reducible to \Smpgcomp{}. Therefore, the claim follows from Corollary~\ref{cor:comp_to_decision}.
\end{proof}

\section{Concluding remarks}

In this paper, we have shown that under a genericity condition on the valuations,
solving feasibility semidefinite problems over the field of Puiseux series reduces
to a well studied class of zero-sum stochastic games. This leads both
to complexity bounds and to algorithms capable experimentally to solve large scale nonarchimedean instances. The interest is also to relate two different
problems which both have unsettled complexities. This is the first
exposition of this approach. 

It would be interesting to relax the current genericity conditions. We believe that finer genericity conditions could involve both the valuations and leading coefficients of the series.

Another interesting question is to use the present approach
to deal with the real case. We already showed in Section~\ref{sec:archimedean} that the nonarchimedean feasibility problem is equivalent to the archimedean one for large values of $t$, with an exponential number of bits. We may ask however whether it is possible to use combinatorial methods to work for smaller values of $t$. 

\section*{Acknowledgments}

We are grateful to the anonymous reviewers for their numerous remarks which helped to improve the presentation of the paper. An abridged version of the present work appeared initially in the ISSAC paper~\cite{issac2016}. We also thank the referees of ISSAC for their detailed comments.\bibliographystyle{alpha}

\appendix

\section{Constructing spectrahedra from mean payoff games}\label{app:equivalence}

Let $\Gamma$ be a stochastic mean payoff game, and let $Q^{(1)}, \dots, Q^{(n)} \in \strop^{m \times m}$ be the matrices as
constructed in the paragraph following Theorem~\ref{theorem:spectra_and_game}. From these matrices, we can define the set $\pertspectra$ as in Definition~\ref{def:sublevel_set}. Denoting by $\shapley$ the Shapley operator of $\Gamma$, we still have:
\begin{lemma}
For any $\lambda \in \R$ we have $\spectra_{\lambda} = \{x \in \trop^{n} \colon \lambda + x \le \shapley(x) \}$.
\end{lemma}

\begin{proof}
Let 
\begin{align*}
(\shapley(x))_k =
\min_{
\substack{a \in A^{(k)}\\
a = \{i, j\}}}
\Bigl(r^a_k + \frac{1}{2}\bigl(\max_{
\substack{
b \in B^{(i)}\\
b = \{l\}}} (r^b_i + x_l) 
+ \max_{
\substack{
b \in B^{(j)}\\
b = \{l\}}} (r^b_j + x_l)\bigr)\Bigr) \,
\end{align*}
denote the Shapley operator of $\Gamma$. As in the proof of Lemma~\ref{lemma:spectrahedron_encodes_shapley}, we have the equivalence
\begin{align*}
\forall k \in [n], \ \lambda + x_{k}  \le (\shapley(x))_k &\iff \\
\forall k \ \forall a \in A^{(k)}, a = \{i,j \}, \ \lambda + x_{k} - r^{a}_{k} \le \frac{1}{2}\bigl(\max_{
\substack{
b \in B^{(i)}\\
b = \{l\}}} (r^b_{i} + x_{l}) 
+ \max_{
\substack{
b \in B^{(j)}\\
b = \{l\}}} (r^b_{j} + x_{l})\bigr) &\iff \\
\forall (i,j), \  \lambda + \max_{\{ a \in \bigcup_{k} A^{(k)} \colon a = \{i,j \} \} }(x_{k} - r^{a}_{k}) \le \frac{1}{2}\bigl(\max_{
\substack{
b \in B^{(i)}\\
b = \{l\}}} (r^b_{i} + x_{l}) 
+ \max_{
\substack{
b \in B^{(j)}\\
b = \{l\}}} (r^b_{j} + x_{l})\bigr)  \, .
\end{align*}
We want to show that the last set of constraints describes $\spectra_{\lambda}$. To do this, recall that an inequality of the form $\max(x,\alpha +y) \ge \max(x', \beta + y)$ is equivalent to $\max(x, \alpha + y) \ge x'$ if $\alpha \ge \beta$, and to $x \ge \max(x', \beta + y)$ if $\beta > \alpha$. Therefore, for every $i \in [m]$ we have the equivalence
\begin{equation}
\begin{aligned}
\lambda + \max_{\{ a \in \bigcup_{k} A^{(k)} \colon a = \{i,i \} \} }(x_{k} - r^{a}_{k}) &\le \max_{
\substack{
b \in B^{(i)}\\
b = \{l\}}} (r^b_{i} + x_{l}) \iff \\
\lambda + \max_{Q^{(k)}_{ii} \in \negtrop}(x_{k} + \abs{Q^{(k)}_{ii}}) &\le \max_{Q^{(l)}_{ii} \in \postrop}(Q^{(l)}_{ii} + x_{l})  \, . \label{eq:linear_constraints}
\end{aligned}
\end{equation}
Moreover, note that if $x \in \trop^{n}$ verifies~(\ref{eq:linear_constraints}), then we have the equality
\begin{equation}
\max_{Q^{(l)}_{ii} \in \postrop}(Q^{(l)}_{ii} + x_{l}) = \max_{
\substack{
b \in B^{(i)}\\
b = \{l\}}} (r^b_{i} + x_{l}) \, . \label{eq:equality_under_linear}
\end{equation}
Indeed, if we have $\max_{Q^{(l)}_{ii} \in \postrop}(Q^{(l)}_{ii} + x_{l}) < r^{\{l^{*}\}}_{i} + x_{l^{*}}$ for some $l^{*} \in [n]$, then by the construction we get $(-r_{l^{*}}^{\{i\}}) > r^{\{l^{*}\}}_{i}$ and $Q^{(l^{*})}_{ii} = \tminus (- r_{l^{*}}^{\{i\}})$. In particular, $\max_{Q^{(k)}_{ii} \in \negtrop}(x_{k} + \abs{Q^{(k)}_{ii}}) > r^{\{l^{*}\}}_{i} + x_{l^{*}}$, what gives a contradiction with~(\ref{eq:linear_constraints}). Furthermore, observe that for any $i \neq j$ we have the equality
\begin{equation}
\lambda + \max_{\{ a \in \bigcup_{k} A^{(k)} \colon a = \{i,j \} \} }(x_{k} - r^{a}_{k}) = \lambda + \max_{Q^{(k)}_{ij} \in \negtrop}(x_{k} + \abs{Q^{(k)}_{ij}}) \, . \label{eq:off_diagonal}
\end{equation}
Suppose that $x \in \spectra_{\lambda}$. Then $x$ verifies~(\ref{eq:linear_constraints}) for all $i \in [m]$. Hence, by~(\ref{eq:equality_under_linear}) and~(\ref{eq:off_diagonal}), for any $i \neq j$ we have
\begin{equation}
\begin{aligned}
\lambda + \max_{\{ a \in \bigcup_{k} A^{(k)} \colon a = \{i,j \} \} }(x_{k} - r^{a}_{k}) &= \lambda + \max_{Q^{(k)}_{ij} \in \negtrop}(x_{k} + \abs{Q^{(k)}_{ij}}) \\
&\le \frac{1}{2}\max_{Q^{(l)}_{ii} \in \postrop}(Q^{(l)}_{ii} + x_{l}) + \frac{1}{2}\max_{Q^{(l)}_{jj} \in \postrop}(Q^{(l)}_{jj} + x_{l}) \\
&= \frac{1}{2}\bigl(\max_{
\substack{
b \in B^{(i)}\\
b = \{l\}}} (r^b_{i} + x_{l}) 
+ \max_{
\substack{
b \in B^{(j)}\\
b = \{l\}}} (r^b_{j} + x_{l})\bigr) \, . \label{eq:equivalence} 
\end{aligned}
\end{equation}
Thus $\lambda + x \le \shapley(x)$. Conversely, if $\lambda + x \le \shapley(x)$, then $x$ also verifies~(\ref{eq:linear_constraints}) for all $i \in [m]$, and the same argument as in~(\ref{eq:equivalence}) shows that $x \in \spectra_{\lambda}$.
\end{proof}

\section{Markov chains}

Let us recall some facts about Markov chains with rewards. We only consider Markov chains on finite spaces. Suppose that we are given a Markov chain $(X_{0}, X_{1}, \dots)$ defined on a finite space $\chainspace$. Recall that such a chain can be described by a \emph{transition matrix} $\trmatrix \in [0,1]^{\card{\chainspace} \times \card{\chainspace}}$, where $\trmatrix_{\state \stateII}$ denotes the probability that chain moves from state $\state$ to state $\stateII$ in one step.

If $\reclass \subset \chainspace$, then we say that $\stdist \in \R_{\ge 0}^{\card{\reclass}}$ is a \emph{stationary distribution} on the set $\reclass$ if $\stdist_{\state} = \sum_{\stateII \in 
\reclass} \stdist_{\stateII}P_{\stateII \state}$ for all $\state \in \reclass$ and $\sum_{\state \in \reclass} \stdist_{\state} = 1$. Furthermore, a set $\reclass \subset \chainspace$ is called a \emph{recurrent class} if it has the following two properties: (a) if $\reclass$ is reached, then the chain will never leave it; (b) if $\reclass$ is reached, then every state in $\reclass$ will be visited infinitely many times (with probability one). Every Markov chain has at least one recurrent class, and every recurrent class has a unique stationary distribution. A state $\state \in \chainspace$ is called \emph{recurrent} if it belongs to a recurrent class. Otherwise, it is called \emph{transient}.

We now introduce Markov chains with payoffs. To this end, with every state $\state \in \chainspace$ we associate a payoff $\payoff_{\state} \in \R$. This quantity is interpreted as follows: there is a controller of the chain, who receives a payoff $\payoff_{\state}$ as soon as the chain leaves the state $\state$. A \emph{(long-term) average payoff} of the controller is defined as
\[
\forall \state \in \chainspace, \, \avpayoff_{\state} = \lim_{N \to \infty} \E \frac{1}{N}\sum_{p = 1}^{N} \payoff_{\state_{p}} \, ,
\]
where the expectation is taken over all trajectories $\state_{1}, \dots, \state_{N}$ starting from $\state_{1} = \state$ in the Markov chain. The next theorem characterizes the average payoff. Before that, let us introduce some additional notation. 

For any state $\state \in \chainspace$, let the random variable $T_{\state} = \inf\{s \ge 1 \colon X_{s} = \state \}$ denote the time of first return to $\state$. By $\theta_{\state}$ we denote the expected time of first return to $\state$, 
\begin{align*}
\theta_{\state} &= \E (T_{\state} | X_{0} = \state) \, .
\end{align*}
Furthermore, let $\xi_{\state}$ be the expected payoff the controlled obtained before returning to $\state$, i.e.,
\[
\xi_{\state} = \E\Bigl(\sum_{s = 0}^{T_\state - 1} \payoff_{X_{s}} \Big | X_{0} = \state \Bigr) \, .
\]

\begin{theorem}\label{theorem:characterization_of_payoff}
If $\state \in \chainspace$ is a fixed initial state, then the average payoff $\avpayoff_{\state}$ is well defined and characterized as follows:
\begin{enumerate}[(i)]
\item\label{point:recurrent_state} Suppose that $\state$ is a recurrent state belonging to the recurrent class $\reclass$. Let $(\stdist_{\stateII})_{\stateII \in \reclass}$ be the stationary distribution on $\reclass$. Then $\stdist_{\state} = 1/\theta_{\state}$. Furthermore, we have
\[
\avpayoff_{\state} = \frac{\xi_{\state}}{\theta_{\state}} = \sum_{\stateII \in \reclass} \payoff_{\stateII}\stdist_{\stateII} \, .
\]
In particular, $\avpayoff_{\state}$ is constant for all states $\state$ belonging to $\reclass$.
\item\label{point:transient_state} If $\state$ is transient and $\reclass_{1}, \dots, \reclass_{p}$ denote all the recurrent classes of the Markov chain, then $\avpayoff_{\state} = \sum_{s = 1}^{p} \avpayoff_{\state_{s}} \psi_{s}$, where, for all $s$, $\psi_{s}$ denotes the probability that the chain starting from $\state$ reaches the recurrent class $\reclass_{s}$, and $\state_{s} \in \reclass_{s}$ is an arbitrary state of $\reclass_{s}$.
\end{enumerate}
\end{theorem}

\begin{remark}\label{remark:payoff_strongly_polynomial}
We point out that given the transition matrix $\trmatrix$, the average payoff $\avpayoff$ can be computed using the algorithm presented in \cite[Appendix A.3 and A.4]{puterman}. This algorithm can be implemented to run in strongly polynomial complexity using the strongly polynomial version of gaussian elimination (presented, for example, in~\cite[Section 1.4]{lovasz_schrijver_geometric_algorithms}).
\end{remark}

Theorem~\ref{theorem:characterization_of_payoff} is well known, and can be easily derived from the analysis of Markov chains presented in the textbook of Chung~\cite[Part I, \S 6, \S 7, and \S 9]{chung_markov_chains}. We give the details for the sake of completeness. Let $\mu_{uw}$ denote the probability that the Markov chain starting from $u \in E$ will reach $w \in E$ at least once, $\mu_{uw} = \Pbb(\exists s \ge 1, X_{s} = w | X_{0} = u)$. By definition, the state $u$ is recurrent if $\mu_{uu} = 1$, and it is transient otherwise.  By $\zeta_{uw}$ we denote the expected number of visits in $w$ before returning to $u$, i.e., 
\begin{align*}
\zeta_{uw} &= \E\Bigl(\sum_{s = 0}^{T_{u} - 1} \ind_{\{X_{s} = w\} } \Big | X_{0} = u\Bigr) \, .
\end{align*}

The following theorem describes the ergodic behavior of any finite (or countable) Markov chain.
\begin{theorem}\label{theorem:ergodic_theorem}
The Cesaro limit 
\[
\lim_{N \to \infty}\frac{1}{N}\sum_{s = 0}^{N}P^{s}
\]
is well defined (we will denote it by $M$). Moreover, the entries of $M$ are given as follows: if $w$ is a transient state, then $M_{uw} = 0$ for all $u$. If $w$ is recurrent, then $M_{uw} = \frac{\mu_{uw}}{\theta_{w}}$ for all $u$.
\end{theorem}
\begin{proof}
See \cite[Part I, \S 6, Theorem 4 and its Corollary]{chung_markov_chains}.
\end{proof}
\begin{remark}
Note that the theorem above does not state that $\theta_{w} < \infty$ if the state $w$ is recurrent. (We work under the convention that $\frac{a}{+\infty} = 0$ for all finite $a$.) Nevertheless, if the chain is finite, then we have $\theta_{w} < \infty$ for all recurrent states $w$, and this can be deduced as a corollary of the theorem above, as discussed below.
\end{remark}

Observe that $M$ is a stochastic matrix (as a limit of stochastic matrices). Moreover, we have $MP = PM = M$. This leads to the following corollary.
\begin{corollary}\label{corollary:stationary_distribution}
If $C \subset E$ is a recurrent class, then $M_{u w} = \frac{1}{\theta_{w}}$ for all $u, w \in C$. Furthermore, $(\pi_{u})_{u \in C}$ defined as $\pi_{u} = \frac{1}{\theta_{u}}$ is the unique stationary distribution on $C$. In particular, if $u \in E$ is a recurrent state, then $\theta_{u} < \infty$.
\end{corollary}
\begin{proof}
Let first claim follows immediately from~Theorem~\ref{theorem:ergodic_theorem}. We will prove that $(\pi_{u})_{u \in C}$ is a stationary distribution on $C$. Let $P_{C}$ denote the square submatrix of $P$ formed by the rows and columns of $P$ with indices in $C$. We define $M_{C}$ analogously. The first claim implies that $M_{C}$ has identical rows. Since $C$ is a recurrent class, we have $P_{uw'} = 0$ for all $u \in C, w' \notin C$. Hence, for all $s$ we have $[P^{s}]_{C} = [P_{C}]^{s}$. Therefore $M_{C} = \lim_{N \to \infty} \frac{1}{N}\sum_{s = 0}^{N}P_{C}^{s}$. Hence, $M_{C}$ is stochastic. In other words, every row of $M_{C}$ is a probability distribution on $C$ and, since $M_{C}P_{C} = M_{C}$, this distribution is a stationary distribution on $C$. Since $C$ is a recurrent class, stationary distribution on $C$ has only strictly positive values. Hence we have $\theta_{u} < \infty$. The fact that the stationary distribution is unique follows from \cite[Part I, \S 7, Theorem 1]{chung_markov_chains}.
\end{proof}

The next theorem characterizes the relationship between entries of $M$ and the values $\zeta_{uw}$.
\begin{theorem}\label{theorem:ratio_of_counting_numbers}
If $(u,w)$ belong to the same recurrent class, then $0 <\zeta_{uw} < \infty$. Moreover, if $(f,g,h,u,w)$ are (not necessarily distinct) states belonging to the same recurrent class, then $M_{fg}/M_{hu} = \zeta_{w g}/\zeta_{w u}$.
\end{theorem}
\begin{proof}
The fact that $0 <\zeta_{uw} < \infty$ follows from \cite[Part I, \S 9, Theorem 2 and 3]{chung_markov_chains}. Moreover, by Corollary~\ref{corollary:stationary_distribution} we have $M_{uw} = \frac{1}{\theta_{w}} > 0$. Hence the claim follows from \cite[Part I, \S 9, Theorem 5 and remarks that precede it]{chung_markov_chains}.
\end{proof}
\begin{corollary}\label{corollary:ratio_of_counting_and_time}
If $(u, w)$ belong to the same recurrent class, then 
\[
M_{uw} = \frac{1}{\theta_{w}} = \frac{\zeta_{uw}}{\theta_{u}} \, .
\]
\end{corollary} 
\begin{proof}
By Theorem~\ref{theorem:ergodic_theorem} and Theorem~\ref{theorem:ratio_of_counting_numbers} we have $\theta_{u}/\theta_{w} = M_{ww}/M_{uu} = \zeta_{u w}/\zeta_{u u}$. By definition $\zeta_{uu} =  1$ and hence $\theta_{u}/\theta_{w} = \zeta_{u w}$.
\end{proof}

\begin{proof}[Proof of Theorem~\ref{theorem:characterization_of_payoff}]
Fix $u \in E$. Observe that for all $N \ge 1$ we have
\[
\E\Bigl(\sum_{s = 0}^{N} r_{X_{s}} \Big | X_{0} = u \Bigr) = [r + Pr + \dots + P^{N}r]_{u} \, .
\]
Therefore $g_{u} = [Mr]_{u}$. In particular, $g_{u}$ is well defined.

Let us suppose that the initial state $u$ is recurrent and denote its recurrent class by $C$. In this case, Corollary~\ref{corollary:stationary_distribution} and Theorem~\ref{theorem:ergodic_theorem} imply that $g_{u} = \sum_{w \in C} r_{w}\pi_{w}$, where $(\pi_{w})_{w \in C}$ is the stationary distribution on $C$. Moreover, Corollary~\ref{corollary:ratio_of_counting_and_time} gives the identity
\begin{align*}
g_{u} &= \frac{1}{\theta_{u}}\sum_{w \in C}\zeta_{uw}r_{w} \\ &= \frac{1}{\theta_{u}}\E\Bigl(\sum_{w \in C}\sum_{s = 0}^{T_{u} - 1} r_{w}\ind_{\{X_{s} = w\} } \Big | X_{0} = u\Bigr) \\
&= \frac{1}{\theta_{u}}\E\Bigl(\sum_{s = 0}^{T_{u} - 1} \sum_{w \in E}r_{w}\ind_{\{X_{s} = w\} } \Big | X_{0} = u\Bigr) \\
&= \frac{1}{\theta_{u}}\E\Bigl(\sum_{s = 0}^{T_{u} - 1} r_{X_{s}} \Big | X_{0} = u\Bigr) = \frac{\xi_{u}}{\theta_{u}} \, .
\end{align*}

Now, suppose that the initial state $u$ is transient. Let $C_{1}, \dots, C_{p}$ denote the recurrent classes in our Markov chain. In this case, Theorem~\ref{theorem:ergodic_theorem} gives the identity
\[
g_{u} = \sum_{s = 1}^{p} \sum_{w \in C_{s}} \frac{\psi_{s}}{\theta_{w}}r_{w} = \sum_{s = 1}^{p} \psi_{s}g_{u_{s}}. \qedhere
\]
\end{proof}

\end{document}